\documentclass[11pt,reqno]{amsart}
\usepackage[a4paper,margin=1in]{geometry}
\usepackage[T1]{fontenc}
\usepackage[utf8]{inputenc}
\usepackage{lmodern}
\usepackage{microtype}
\usepackage{amsmath,amssymb,amsfonts,amsthm,mathtools}
\usepackage{aliascnt}
\usepackage{bm}
\usepackage{enumitem}
\usepackage{xcolor}
\usepackage{tikz}
\usepackage[colorlinks=true,linkcolor=blue,citecolor=blue,urlcolor=blue]{hyperref}
\usepackage[nameinlink,capitalize,noabbrev]{cleveref}
\usepackage{graphicx}

\makeatletter
\@namedef{subjclassname@2020}{\textup{2020} Mathematics Subject Classification}
\makeatother

\numberwithin{equation}{section}
\allowdisplaybreaks

\theoremstyle{plain}
\newtheorem{theorem}{Theorem}[section]

\newaliascnt{proposition}{theorem}
\newtheorem{proposition}[proposition]{Proposition}
\aliascntresetthe{proposition}

\newaliascnt{lemma}{theorem}
\newtheorem{lemma}[lemma]{Lemma}
\aliascntresetthe{lemma}

\newaliascnt{corollary}{theorem}
\newtheorem{corollary}[corollary]{Corollary}
\aliascntresetthe{corollary}

\theoremstyle{definition}
\newaliascnt{assumption}{theorem}
\newtheorem{assumption}[assumption]{Assumption}
\aliascntresetthe{assumption}

\newaliascnt{definition}{theorem}

\aliascntresetthe{definition}

\theoremstyle{remark}
\newaliascnt{remark}{theorem}
\newtheorem{remark}[remark]{Remark}
\aliascntresetthe{remark}

\crefname{assumption}{Assumption}{Assumptions}
\crefname{lemma}{Lemma}{Lemmas}
\crefname{proposition}{Proposition}{Propositions}
\crefname{theorem}{Theorem}{Theorems}
\crefname{corollary}{Corollary}{Corollaries}
\crefname{remark}{Remark}{Remarks}
\crefname{definition}{Definition}{Definitions}
\Crefname{assumption}{Assumption}{Assumptions}
\Crefname{lemma}{Lemma}{Lemmas}
\Crefname{proposition}{Proposition}{Propositions}
\Crefname{theorem}{Theorem}{Theorems}
\Crefname{corollary}{Corollary}{Corollaries}
\Crefname{remark}{Remark}{Remarks}
\Crefname{definition}{Definition}{Definitions}

\newcommand{\R}{\mathbb R}
\newcommand{\E}{\mathbb E}
\newcommand{\Prob}{\mathbb P}
\newcommand{\Var}{\operatorname{Var}}
\newcommand{\diag}{\operatorname{diag}}
\newcommand{\Tr}{\operatorname{Tr}}
\newcommand{\dd}{\,\mathrm d}
\newcommand{\Ai}{\operatorname{Ai}}
\newcommand{\one}{\mathbf 1}
\newcommand{\bxi}{\bm\xi}
\newcommand{\cF}{\mathcal F}
\newcommand{\cG}{\mathcal G}
\newcommand{\eps}{\varepsilon}
\newcommand{\limn}{\lim_{N\to\infty}}
\newcommand{\doi}[1]{\href{https://doi.org/#1}{\nolinkurl{doi:#1}}}
\newcommand{\rev}[1]{\textcolor{red}{#1}}

\title[Moderate Deviations for Deformed GUE]{Moderate Deviations for the Largest Eigenvalue of a Randomly Deformed Gaussian Unitary Ensemble}

\author{Shaochen Wang}
\address[S. Wang]{School of Mathematics, South China University of Technology, Guangzhou, China}
\email{mascwang@scut.edu.cn}

\author{Guangyu Yang}
\address[G. Yang]{School of Mathematics and Statistics, Zhengzhou University, Zhengzhou, China}
\email{guangyu@zzu.edu.cn}

\date{\today}



\begin{document}

\begin{abstract}
We study moderate deviations for the largest eigenvalue of the randomly deformed Gaussian unitary ensemble introduced by Johansson (Probab. Theory Relat. Fields, {\bf 138}: 75--112, 2007).  In the fixed-coupling regime, the rescaled largest eigenvalue converges to the convolution of the Tracy--Widom law and a Gaussian law arising from the random displacement of the spectral edge.  We derive sharp logarithmic asymptotics for the right and left tails on growing Airy scales and obtain the corresponding phase diagram.  The two tails have different transition scales.  At criticality, the rate functions are nontrivial infimal convolutions of the Tracy--Widom and Gaussian rate functions.  We establish a trace-norm Airy approximation that is uniform over typical diagonal configurations on a logarithmic window.  A conditional convolution argument then combines the resulting tail estimates with the Gaussian moderate deviations of the edge displacement. 
\end{abstract}

\maketitle

\medskip

\setcounter{tocdepth}{1}
\tableofcontents


\section{Introduction}\label{sec:introduction}


Additive deformations of random matrices provide a basic framework for studying how a structured or disordered background changes a random spectrum. Such models are constructed by adding a deterministic or random Hermitian matrix to a standard random matrix. A finite-rank deformation represents a small number of distinguished directions, whereas a full-rank deformation describes an inhomogeneous background. The deformation may shift a regular spectral edge, create outliers, or change the fluctuation law of the extreme eigenvalues. These questions arise frequently in high-dimensional statistics and mathematical physics, where the largest eigenvalue may be governed either by the deformation or by the collective behavior of the random matrix.

Finite-rank and full-rank deformations lead to different edge behaviors.  Finite-rank sources can create outliers and phase transitions at the edge.  This line includes the Baik--Ben Arous--P\'ech\'e transition for spiked sample covariance matrices \cite{BaikBenArousPeche2005}, the corresponding transition for Hermitian additive deformations \cite{Peche2006}, and extensions to other symmetry classes \cite{BloemendalVirag2013}.  For full-rank deformations, the study of the global spectral law goes back to Pastur \cite{Pastur1972}, while fixed-scale edge statistics and edge universality have been established in several settings \cite{Shcherbina2011,CapitainePeche2016,LeeSchnelli2015}.  A prescribed source has a deterministic deformed edge.  When the full-rank source is random, fluctuations of its empirical spectral measure can displace the edge from sample to sample.  This additional source of fluctuation is central to the present work.

Large deviations of extreme eigenvalues form a complementary, macroscopic theory.  They have been studied for rank-one and random finite-rank deformations \cite{Maida2007,BenaychGeorgesGuionnetMaida2012}, prescribed full-rank deformations \cite{McKenna2021}, and sums of random matrices \cite{GuionnetMaida2020}.  A recent work treats the smallest eigenvalue of a deformed Gaussian orthogonal ensemble with an outlier \cite{BoursierGuionnet2026}.  These developments build on large deviations for empirical spectral measures \cite{BenArousGuionnet1997} and asymptotics of spherical integrals \cite{GuionnetMaida2005}.  A replica analysis for deformed Gaussian ensembles was given in \cite{LeDoussal2025}.  Together with the fixed-scale edge results, this literature leaves an intermediate regime in which the Airy argument diverges but remains below the macroscopic scale.  That regime is particularly sensitive to a random full-rank source.

In this paper, we consider the randomly deformed Gaussian unitary ensemble (GUE) studied by Johansson \cite{Johansson2007}.  Let
\begin{equation}\label{eq:model}
 M_N=\diag(\xi_1,\ldots,\xi_N)+\sqrt{2S_N}\,V_N,
 \qquad S_N=\frac{\alpha_N^2}{N^{2/3}},
\end{equation}
where $\xi_1,\ldots,\xi_N$ are independent and identically distributed (i.i.d.) centered random variables with variance $\sigma^2\in(0,\infty)$ and satisfy a Cram\'{e}r condition.  The matrix $V_N$ is an independent GUE matrix with density proportional to $\exp\{-\Tr V_N^2\}\dd V_N$, and $\alpha_N>0$ measures the strength of the GUE coupling at the edge scale.  We write $\lambda_N=\lambda_{\max}(M_N)$.  Up to normalization, the model is a one-time marginal of Dyson Brownian motion \cite{Dyson1962} and a unitary version of the Rosenzweig--Porter ensemble for disordered Hamiltonians \cite{RosenzweigPorter1960}. For fixed $\alpha_N\equiv\alpha>0$, Johansson \cite[Theorem~1.12]{Johansson2007} proved that there is a deterministic centering $R_N^{\mathrm J}\sim2\alpha N^{1/6}$ such that
\begin{equation}\label{eq:Johansson-limit}
 \frac{\sqrt N}{\alpha}\bigl(\lambda_N-R_N^{\mathrm J}\bigr)
\stackrel{d}{\longrightarrow} X_2+G_\alpha.
\end{equation}
Here $X_2$ has the GUE Tracy--Widom distribution $F_2$, and $G_\alpha$ is centered Gaussian with variance $\sigma^2/\alpha^2$, independent of $X_2$.  Throughout, we use the normalization
\begin{equation}\label{eq:F2-definition}
 F_2(t)=\det(I-K_{\Ai})_{L^2(t,\infty)},
 \qquad
 K_{\Ai}(x,y)=\int_0^\infty
 \Ai(x+\lambda)\Ai(y+\lambda)\,\dd\lambda,
\end{equation}
where $\Ai$ is the Airy function.

The convolution in \eqref{eq:Johansson-limit} can be understood as follows.  Conditional on the random diagonal, the GUE perturbation produces a Tracy--Widom fluctuation about the empirical deformed edge.  The random source also displaces that edge; to leading order, the displacement is a normalized sum of the diagonal entries and is therefore Gaussian.  Lee and Schnelli \cite[Theorems~2.8 and~2.11]{LeeSchnelli2015} proved Tracy--Widom universality after centering at the empirical deformed edge for broad classes of deterministic or random diagonal potentials.  For bounded i.i.d. potentials, they also obtained an analogous convolution of the Tracy--Widom and Gaussian laws in their critical fluctuation regime.

The weak convergence in \eqref{eq:Johansson-limit} describes bounded arguments on the Airy scale, but it does not determine tail probabilities when the argument diverges.  The relevant Tracy--Widom tail asymptotics, obtained in the analyses of Tracy and
Widom~\cite{TracyWidom1994}, Baik, Buckingham, and
DiFranco~\cite{BaikBuckinghamDiFranco2008}, and Deift, Its, and
Krasovsky~\cite{DeiftItsKrasovsky2008}, are
\begin{align}
 \log\bigl(1-F_2(x)\bigr)
 &=-\frac43x^{3/2}+O(\log x),
 &&x\to\infty,                                      \label{eq:TW-right-intro}\\
 \log F_2(-x)
 &=-\frac1{12}x^3+O(\log x),
 &&x\to\infty.                                      \label{eq:TW-left-intro}
\end{align}
The right-tail exponent was extended to general Tracy--Widom-$\beta$ laws by Dumaz
and Vir\'ag~\cite{DumazVirag2013}.  Quantitative soft-edge deviation estimates
were obtained by Ledoux and Rider~\cite{LedouxRider2010} for $\beta$-ensembles
and by Erd\H{o}s and Xu~\cite{ErdosXu2023} for Wigner matrices.  Huang,
Landon, and Yau~\cite{HuangLandonYau2020} found a Tracy--Widom--Gaussian
transition for sparse Erd\H{o}s--R\'enyi matrices.  These results show that
soft-edge deviations and Gaussian edge displacements can coexist.

Our aim is to determine the logarithmic right- and left-tail asymptotics when the Airy scale grows and the coupling $\alpha_N$ may vary with $N$.  Let $a_N\to\infty$ denote the moderate scale.  The random edge displacement has a quadratic Gaussian rate.  On the right, the Tracy--Widom exponent is of order $a_N^{3/2}$ while the Gaussian displacement exponent is of order $\alpha_N^2a_N^2$; on the left, the corresponding orders are $a_N^3$ and $\alpha_N^2a_N^2$.  Since the Tracy--Widom tails have different powers, the right and left tails have different transition parameters and critical coupling scales.

Our main result, \Cref{thm:main}, gives the phase diagram for the right and left tails; see \Cref{fig:phase-diagram}.  At the critical scales, the rate functions are infimal convolutions of the Tracy--Widom and Gaussian rate functions.  In particular, for fixed coupling, the right moderate tail has Tracy--Widom asymptotics, whereas the left moderate tail has Gaussian asymptotics. The main technical issue is uniformity over the random diagonal: its empirical resolvent determines both the conditional kernel and the conditional edge.  For the left tail, probabilities can be as small as $\exp\{-c a_N^3\}$, so the approximation error must be negligible at that scale.  After a superexponentially accurate truncation, a resolvent expansion yields the Gaussian MDP for the edge displacement.  A uniform steepest-descent analysis of Johansson's kernel gives a trace-norm Airy approximation and conditional Tracy--Widom tail estimates on a logarithmic window.  A conditional convolution argument then combines the two estimates.

The rest of the paper is organized as follows.
\Cref{sec:main} states the assumptions, the deterministic centering, and the main theorem.  \Cref{sec:truncation,sec:shift} treat the truncation, the population saddle, and the Gaussian MDP for the edge displacement.  \Cref{sec:kernel,sec:airy,sec:conditional-tails} establish the conditional Airy approximation and tail estimates.  \Cref{sec:phase-proof} proves the phase diagram, and \Cref{sec:conclude-rmk} gives concluding remarks.  The proof of \Cref{lem:shifted-contours} is deferred to \Cref{app:shifted-contours}.


\section{Main results}\label{sec:main}


Throughout the paper, all logarithms are natural.  The symbols $c$ and $C$ denote positive constants whose values may change from line to line; their dependence on fixed model parameters is suppressed.  All asymptotic statements are understood as $N\to\infty$, and $\one_E$ denotes the indicator of an event $E$.  If $\upsilon_N\to\infty$, we say that events $E_N$ are {\it superexponentially} small at speed $\upsilon_N$ when
\[
 \limsup_{N\to\infty}\frac1{\upsilon_N}\log\Prob(E_N)=-\infty.
\]

We first impose three standing assumptions.


\subsection{Assumptions and centering}


\begin{assumption}\label{ass:xi}
The variables $\xi_1,\xi_2,\ldots$ are i.i.d., independent of $V_N$, and satisfy
\[
 \E\xi_1=0,
 \qquad
 \Var(\xi_1)=\sigma^2\in(0,\infty).
\]
There exists $\kappa>0$ such that
\begin{equation}\label{eq:exp-moment}
 \E e^{\kappa|\xi_1|}<\infty.
\end{equation}
\end{assumption}

This Cram\'er condition provides the uniform moment generating function expansion needed for the Gaussian moderate deviations and makes the truncation error superexponentially small.

\begin{assumption}\label{ass:alpha}
The positive sequence $\alpha_N$ satisfies
\begin{equation}\label{eq:alpha-subpoly}
 \frac{|\log\alpha_N|}{\log N}\longrightarrow0.
\end{equation}
\end{assumption}

Thus $\alpha_N=N^{o(1)}$ and $\alpha_N^{-1}=N^{o(1)}$.  The coupling may vanish, remain bounded away from zero and infinity, or diverge, but it cannot vary at a fixed polynomial rate.

\begin{assumption}\label{ass:scale}
The moderate scale $a_N$ satisfies
\begin{equation}\label{eq:scale-assumptions}
 a_N\to\infty,
 \qquad
 a_N^3=o(\log N),
 \qquad
 \alpha_Na_N\to\infty.
\end{equation}
\end{assumption}

The last condition ensures that the random edge displacement is observed on a moderate deviation scale.  Together with \eqref{eq:alpha-subpoly}, these assumptions imply $\alpha_Na_N=N^{o(1)}=o(\sqrt N)$.

\medskip

We now define the deterministic edge used throughout the paper.  Fix
\begin{equation}\label{eq:gamma-delta}
 \gamma_0=\frac1{24},
 \qquad
 0<\delta<\frac1{100},
 \qquad
 \tau_N=N^{\gamma_0}.
\end{equation}
Let $\mu$ denote the law of $\xi_1$, and set $p_N=\mu([-\tau_N,\tau_N])$.  Since $p_N\to1$, there is an $N_0$ such that $p_N>0$ for every $N\ge N_0$.  All constructions below are made for $N\ge N_0$; the finitely many earlier rows are irrelevant to the asymptotic statements.  For every Borel set $B\subset\R$, define the conditioned truncation of $\mu$ by
\begin{equation}\label{eq:truncated-law}
 \mu_N(B)=
 \frac{\mu(B\cap[-\tau_N,\tau_N])}
      {p_N}.
\end{equation}
Define
\begin{equation}\label{eq:GN}
 G_N(z)=\int\frac{1}{z-t}\,\dd\mu_N(t),
 \qquad z\in\mathbb{C}\setminus[-\tau_N,\tau_N],
\end{equation}
and
\begin{equation}\label{eq:AN}
 A_N=\alpha_NN^{1/6}.
\end{equation}
Since $A_N/\tau_N=N^{1/8+o(1)}$, \Cref{lem:population-saddle} shows that, for all sufficiently large $N$, there is a unique $w_N>2\tau_N$ satisfying
\begin{equation}\label{eq:population-saddle}
 \int\frac{1}{(w_N-t)^2}\,\dd\mu_N(t)
 =\frac1{A_N^2}.
\end{equation}
Set
\begin{equation}\label{eq:RN}
 R_N=w_N+A_N^2G_N(w_N).
\end{equation}
Then $w_N\sim A_N$ and $R_N\sim2A_N$.  Recalling that $\lambda_N=\lambda_{\max}(M_N)$, define
\begin{equation}\label{eq:ZN}
 Z_N=\frac{\sqrt N}{\alpha_N}(\lambda_N-R_N).
\end{equation}
The centering $R_N$ is deterministic and depends only on the law of $\xi_1$, on $\alpha_N$, and on $N$.  Its definition through the conditioned law $\mu_N$ is part of the superexponentially accurate truncation scheme.

Whenever a full MDP is invoked below, it is understood in the standard open-set/closed-set sense, and a rate function is called good when its sublevel sets are compact.  Because the right and left tails generally have different speeds, the main theorem states two families of one-sided moderate deviation asymptotics.


\subsection{Phase diagram and main theorem}


For $\theta\in(0,\infty)$ and $x>0$, define
\begin{align}
 J_+^{(\theta)}(x)
 &=\inf_{0\le u\le x}
 \left\{\frac43u^{3/2}
 +\frac{\theta(x-u)^2}{2\sigma^2}\right\},
 \label{eq:Jplus}\\
 J_-^{(\theta)}(x)
 &=\inf_{0\le u\le x}
 \left\{\frac1{12}u^3
 +\frac{\theta(x-u)^2}{2\sigma^2}\right\}.
 \label{eq:Jminus}
\end{align}
Each objective is strictly convex on $(0,x)$.  Its derivative is negative at $u=0$ and positive at $u=x$, so the minimizer is unique and belongs to $(0,x)$.  Solving the Euler equations gives
\begin{align}
 u_+^*(x,\theta)
 &=\left(
 \frac{\sqrt{\sigma^4+\theta^2x}-\sigma^2}{\theta}
 \right)^2,
 \label{eq:uplus-star}\\
 u_-^*(x,\theta)
 &=\frac{2}{\sigma^2}
 \left(\sqrt{\theta^2+\theta\sigma^2x}-\theta\right).
 \label{eq:uminus-star}
\end{align}
These formulas also show directly that $0<u_+^*(x,\theta),u_-^*(x,\theta)<x$.  The two transition parameters are
\begin{equation}\label{eq:theta-parameters}
 \theta_N^+=\alpha_N^2a_N^{1/2},
 \qquad
 \theta_N^-=\frac{\alpha_N^2}{a_N}.
\end{equation}
For each fixed $x>0$, the critical rate functions connect continuously to the pure regimes:
\begin{align*}
 \lim_{\theta\to\infty}J_+^{(\theta)}(x)&=\frac43x^{3/2},&
 \lim_{\theta\to\infty}J_-^{(\theta)}(x)&=\frac1{12}x^3,\\
 \lim_{\theta\downarrow0}\frac{J_+^{(\theta)}(x)}{\theta}
 &=\frac{x^2}{2\sigma^2},&
 \lim_{\theta\downarrow0}\frac{J_-^{(\theta)}(x)}{\theta}
 &=\frac{x^2}{2\sigma^2}.
\end{align*}
Here $\theta_N^+$ compares the Gaussian speed with the right Tracy--Widom speed, whereas $\theta_N^-$ makes the corresponding comparison on the left.  The resulting phase diagram is shown in \Cref{fig:phase-diagram}.

\begin{figure}[t]
\centering
\begin{tikzpicture}[x=1.08cm,y=1cm]
  \def\tone{2.6}   
  \def\tfix{5.0}   
  \def\ttwo{7.4}   
  \def\xb{0.35}    

  \node[anchor=east,font=\small] at (-0.15,1.5) {right tail};
  \fill[orange!20] (0,1.2) rectangle (\tone-\xb,1.8);
  \fill[purple!22] (\tone-\xb,1.2) rectangle (\tone+\xb,1.8);
  \fill[blue!18] (\tone+\xb,1.2) rectangle (10,1.8);
  \node[font=\small] at (0.95,1.5) {Gaussian};
  \node[font=\small] at (\tone,1.5) {crossover};
  \node[font=\small] at (8.0,1.5) {Tracy--Widom};

  \node[anchor=east,font=\small] at (-0.15,0.3) {left tail};
  \fill[orange!20] (0,-0.0) rectangle (\ttwo-\xb,0.6);
  \fill[purple!22] (\ttwo-\xb,0.0) rectangle (\ttwo+\xb,0.6);
  \fill[blue!18] (\ttwo+\xb,0.0) rectangle (10,0.6);
  \node[font=\small] at (3.4,0.3) {Gaussian};
  \node[font=\small] at (\ttwo,0.3) {crossover};
  \node[font=\small] at (9.3,0.3) {Tracy--Widom};

  \draw[dashed,gray] (\tfix,-0.28) -- (\tfix,1.8);

  \draw[->] (0,-0.42) -- (10.35,-0.42) node[right,font=\small] {$\alpha_N$};
  \draw (\tone,-0.49) -- (\tone,-0.35);
  \node[below,font=\small] at (\tone,-0.50) {$\alpha_N\asymp a_N^{-1/4}$};
  \draw (\tfix,-0.49) -- (\tfix,-0.35);
  \node[below,font=\small,align=center] at (\tfix,-0.50)
    {$\alpha_N\asymp1$\\[-2pt]{\footnotesize(fixed coupling)}};
  \draw (\ttwo,-0.49) -- (\ttwo,-0.35);
  \node[below,font=\small] at (\ttwo,-0.50) {$\alpha_N\asymp a_N^{1/2}$};
  \node[below left,font=\footnotesize,align=right] at (0,-0.46)
    {$\alpha_Na_N\to\infty$\\ (assumed)};
  \node[below right,font=\footnotesize] at (10,-0.46) {$\alpha_N\to\infty$ allowed};
\end{tikzpicture}
\caption{\scriptsize The phase diagram on a common (logarithmic) coupling axis.  The right transition occurs when $\alpha_N\asymp a_N^{-1/4}$, whereas the left transition occurs at the much larger scale $\alpha_N\asymp a_N^{1/2}$.  Between the two thresholds the right tail is already Tracy--Widom while the left tail is still Gaussian; this intermediate band contains every fixed coupling $\alpha_N\equiv\alpha$ (dashed line), which gives \Cref{cor:fixed-alpha}.  The Gaussian regimes have speed $\alpha_N^2a_N^2$, the Tracy--Widom regimes have speeds $a_N^{3/2}$ on the right and $a_N^3$ on the left, and the crossover bands are the windows $\theta_N^+\asymp1$ and $\theta_N^-\asymp1$, with critical rate functions $J_+^{(\theta)}$ and $J_-^{(\theta)}$.}
\label{fig:phase-diagram}
\end{figure}
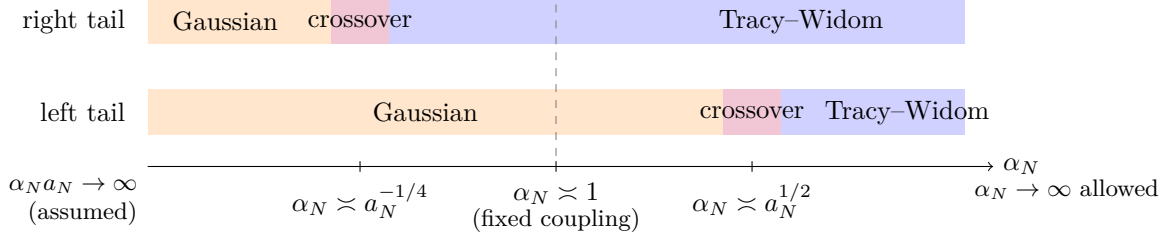

\begin{theorem}\label{thm:main}
Suppose \Cref{ass:xi,ass:alpha,ass:scale} hold.  Then, for every fixed $x>0$, the following limits hold.

\medskip
\noindent\textbf{Right tail.}
\begin{enumerate}[label=\textup{(R\arabic*)},leftmargin=3.2em]
\item If $\theta_N^+\to\infty$, then
\begin{equation}\label{eq:R1}
 \limn\frac1{a_N^{3/2}}
 \log\Prob(Z_N>a_Nx)=-\frac43x^{3/2}.
\end{equation}

\item If $\theta_N^+\to\theta\in(0,\infty)$, then
\begin{equation}\label{eq:R2}
 \limn\frac1{a_N^{3/2}}
 \log\Prob(Z_N>a_Nx)=-J_+^{(\theta)}(x).
\end{equation}

\item If $\theta_N^+\to0$, then
\begin{equation}\label{eq:R3}
 \limn\frac1{\alpha_N^2a_N^2}
 \log\Prob(Z_N>a_Nx)=-\frac{x^2}{2\sigma^2}.
\end{equation}
\end{enumerate}

\medskip
\noindent\textbf{Left tail.}
\begin{enumerate}[label=\textup{(L\arabic*)},leftmargin=3.2em]
\item If $\theta_N^-\to\infty$, then
\begin{equation}\label{eq:L1}
 \limn\frac1{a_N^3}
 \log\Prob(Z_N\le-a_Nx)=-\frac1{12}x^3.
\end{equation}

\item If $\theta_N^-\to\theta\in(0,\infty)$, then
\begin{equation}\label{eq:L2}
 \limn\frac1{a_N^3}
 \log\Prob(Z_N\le-a_Nx)=-J_-^{(\theta)}(x).
\end{equation}

\item If $\theta_N^-\to0$, then
\begin{equation}\label{eq:L3}
 \limn\frac1{\alpha_N^2a_N^2}
 \log\Prob(Z_N\le-a_Nx)=-\frac{x^2}{2\sigma^2}.
\end{equation}
\end{enumerate}
\end{theorem}

In the critical rate functions, the variable $u$ represents the part of the deviation carried by the conditional Tracy--Widom fluctuation, while $x-u$ is carried by the random displacement of the edge.  The infimum selects the most likely division between these two sources of fluctuation.  The explicit formulas \eqref{eq:uplus-star} and \eqref{eq:uminus-star} show that this division is interior in the crossover regime.

The fixed-coupling case illustrates the asymmetry most directly.

\begin{corollary}\label{cor:fixed-alpha}
Under the assumptions of \Cref{thm:main}, if $\alpha_N\equiv\alpha\in(0,\infty)$, then, for every fixed $x>0$,
\begin{align}
 \limn\frac1{a_N^{3/2}}
 \log\Prob(Z_N>a_Nx)
 &=-\frac43x^{3/2},                                      \label{eq:fixed-right}\\
 \limn\frac1{a_N^2}
 \log\Prob(Z_N\le-a_Nx)
 &=-\frac{\alpha^2x^2}{2\sigma^2}.                      \label{eq:fixed-left}
\end{align}
\end{corollary}


\begin{remark}\label{rmk:phases}
The thresholds in \eqref{eq:theta-parameters} quantify the asymmetry seen in \Cref{fig:phase-diagram}: a coupling that tends to zero may still yield a Tracy--Widom right tail, and a coupling that tends to infinity may still yield a Gaussian left tail when it grows more slowly than $a_N^{1/2}$.
\end{remark}

\begin{remark}\label{rem:log-window-statement}
The restriction $a_N^3=o(\log N)$ is dictated by the left tail.  Indeed, the proof compares the finite-$N$ conditional Fredholm determinant with $F_2(-a_Nx)$, which is of order $\exp\{-a_N^3x^3/12+O(\log a_N)\}$ for fixed $x>0$.  A polynomially small relative determinant error is therefore negligible only on a {\it logarithmic Airy window}---that is, a spectral interval of size $O((\log N)^{1/3})$.  Extending the theorem to a polynomial window would require an exponentially accurate finite-$N$ expansion; see \Cref{rem:log-window}.
\end{remark}


\section{Truncation and the deterministic edge}\label{sec:truncation}


The centering in \eqref{eq:RN} is defined through the conditioned law
$\mu_N$, not through the original law $\mu$.  We first justify this
replacement at the level of the largest eigenvalue.  The purpose of the
truncation is twofold: it gives a deterministic support bound for every
later contour argument, and it rules out poles close to the population
saddle.  The truncated variables have a superpolynomially small first
moment, which is retained exactly in the expansion of \Cref{sec:shift}.

Recall from \eqref{eq:gamma-delta} that
$\tau_N=N^{\gamma_0}$ with $\gamma_0=1/24$.  This choice is convenient:
besides $\tau_N=o(A_N)$ and domination of the $N^{o(1)}$
moderate deviation speeds, the estimates below require
$6\gamma_0<1/2-\delta$ in \Cref{lem:good-event} and $7\gamma_0<1/2$ in
\Cref{prop:shift-mdp}.  With $0<\delta<1/100$, every fixed exponent
$0<\gamma_0<1/14$ satisfies these conditions; the choice $\gamma_0=1/24$
lies in this range.


\subsection{A superexponentially accurate truncation coupling}


Define
\begin{equation}\label{eq:tail-probability}
 q_N=\mu\bigl(\{t:|t|>\tau_N\}\bigr),
 \qquad
 p_N=1-q_N.
\end{equation}
By \eqref{eq:truncated-law}, the conditioned law $\mu_N$ is supported on
$[-\tau_N,\tau_N]$.  The exponential moment assumption
\eqref{eq:exp-moment} and Markov's inequality give
\begin{equation}\label{eq:xi-tail-bound}
 q_N\le \E e^{\kappa|\xi_1|}\,e^{-\kappa\tau_N}
 =O(e^{-\kappa\tau_N}).
\end{equation}

Let $\widehat\xi_{1,N},\ldots,\widehat\xi_{N,N}$ be i.i.d.\ with law
$\mu_N$.  For each $j$, independently of the other coordinates, couple
$\widehat\xi_{j,N}$ maximally with $\xi_j$; see, for example,
Lindvall~\cite{Lindvall1992}.  For the total variation (TV) distance, we use the convention
$d_{\mathrm{TV}}(\nu_1,\nu_2)=\sup_B|\nu_1(B)-\nu_2(B)|$, where the
supremum ranges over Borel sets,
\begin{equation}\label{eq:maximal-coupling}
 \Prob(\widehat\xi_{j,N}\ne\xi_j)
 =d_{\mathrm{TV}}(\mu_N,\mu)
 =q_N.
\end{equation}
Indeed, $\mu_N$ equals $p_N^{-1}\mu$ on Borel subsets of
$[-\tau_N,\tau_N]$ and vanishes outside that interval; the largest
discrepancy is attained by the complement of the interval.

Using the same GUE matrix $V_N$ as in \eqref{eq:model}, define
\[
 \widehat M_N=
 \diag(\widehat\xi_{1,N},\ldots,\widehat\xi_{N,N})
 +\sqrt{2S_N}\,V_N,
\]
and let $\widehat\lambda_N$ denote its largest eigenvalue.  The following
lemma says that this replacement is invisible at every exponential scale
appearing in the main theorem.

\begin{lemma}\label{lem:truncation-coupling}
There are constants $c,C>0$ such that
\begin{equation}\label{eq:truncation-coupling-prob}
 \Prob(\widehat\lambda_N\ne\lambda_N)
 \le CN e^{-c\tau_N}.
\end{equation}
Consequently, for
\[
 \upsilon_N\in
 \left\{a_N^{3/2},\,a_N^3,\,\alpha_N^2a_N^2\right\},
\]
one has
\begin{equation}\label{eq:truncation-speed}
 \limsup_{N\to\infty}\frac1{\upsilon_N}
 \log\Prob(\widehat\lambda_N\ne\lambda_N)=-\infty.
\end{equation}
In particular, at each of these speeds, the scaled variables
$\sqrt N(\widehat\lambda_N-R_N)/\alpha_N$ and $Z_N$ are exponentially
equivalent.
\end{lemma}

\begin{proof}
A union bound over the maximal coordinate couplings
\eqref{eq:maximal-coupling}, together with \eqref{eq:xi-tail-bound},
gives
\[
 \Prob\bigl(\exists j\le N:
 \widehat\xi_{j,N}\ne\xi_j\bigr)
 \le Nq_N
 \le CN e^{-c\tau_N}.
\]
On the complementary event, all diagonal entries coincide.  Since the
same matrix $V_N$ is used in both models, the two full matrices, and
therefore their largest eigenvalues, coincide as well.  This proves
\eqref{eq:truncation-coupling-prob}.

Under \Cref{ass:alpha,ass:scale}, each displayed
speed is $N^{o(1)}$, whereas $\tau_N=N^{1/24}$ and the factor $N$ is
absorbed by a further factor $e^{-c'\tau_N}$.  Hence
\eqref{eq:truncation-speed} follows.  Finally, the two scaled largest
eigenvalues differ only on the event in
\eqref{eq:truncation-coupling-prob}, so their exponential equivalence is
immediate.
\end{proof}

From now until the final transfer back to the original matrix $M_N$,
probabilities and expectations refer to the truncated triangular array
unless another measure is displayed explicitly, and we suppress the hats.
Thus, in the $N$th row, the variables denoted by $\xi_j$ are i.i.d. with
law $\mu_N$ and satisfy $|\xi_j|\le\tau_N$ almost surely.  We write
$\bxi=(\xi_1,\ldots,\xi_N)$ for the disorder configuration in the $N$th row.


\subsection{The population saddle}


For $k\ge1$, define the moments of the conditioned law by
\[
 m_{k,N}=\int t^k\,\dd\mu_N(t).
\]
For every fixed integer $k\ge1$,
\begin{equation}\label{eq:tail-moment-bound}
 \int_{|t|>\tau_N}|t|^k\,\dd\mu(t)
 \le C_ke^{-c\tau_N}.
\end{equation}
Indeed, for all sufficiently large $N$,
\begin{align*}
 \int_{|t|>\tau_N}|t|^k\,\dd\mu(t)
\le
 \left(\int e^{\kappa |t|/2}\,\dd\mu(t)\right)
 \sup_{u>\tau_N}u^ke^{-\kappa u/2}
\le C_ke^{-c\tau_N}.
\end{align*}
Since $\int t\,\dd\mu(t)=0$,
\begin{equation}\label{eq:m1-tail}
 m_{1,N}
 =-p_N^{-1}\int_{|t|>\tau_N}t\,\dd\mu(t).
\end{equation}
Similarly,
\begin{equation}\label{eq:m2-tail}
 m_{2,N}-\sigma^2
 =p_N^{-1}\bigg\{(1-p_N)\sigma^2
 -\int_{|t|>\tau_N}t^2\,\dd\mu(t)\bigg\}.
\end{equation}
Using \eqref{eq:tail-moment-bound}, $p_N\to1$, and the uniform absolute
moment bound under $\mu$, we obtain, for every fixed $k\ge1$,
\begin{equation}\label{eq:truncated-moments}
 m_{1,N}=O(e^{-c\tau_N}),
 \qquad
 m_{2,N}=\sigma^2+O(e^{-c\tau_N}),
 \qquad
 \sup_{N\ge N_0}\int |t|^k\,\dd\mu_N(t)<\infty.
\end{equation}
Thus truncation changes the first two moments only by a
superpolynomially small amount, while giving uniform moments of every
fixed order.

We now construct the deterministic saddle and derive its expansion.  The
leading asymptotics $w_N\sim A_N$ are all that is needed to control the
contours, but the next term in $R_N$ is also identified below in order
to make the deterministic centering completely explicit.

\begin{lemma}
\label{lem:population-saddle}
For all sufficiently large $N$, equation \eqref{eq:population-saddle}
has a unique solution $w_N>2\tau_N$.  Moreover,
\begin{align}
 \frac{w_N}{A_N}
 &=1+O(A_N^{-2})+O(e^{-c\tau_N}),                      \label{eq:w-asymptotic}\\
 R_N
 &=2A_N+O(A_N^{-1})+O(e^{-c\tau_N}).                  \label{eq:R-asymptotic}
\end{align}
More precisely,
\begin{equation}\label{eq:wR-precise}
 w_N=A_N+\frac{3\sigma^2}{2A_N}+o(A_N^{-1}),
 \qquad
 R_N=2A_N+\frac{\sigma^2}{A_N}+o(A_N^{-1}).
\end{equation}
In particular, $w_N\sim A_N$ and $R_N\sim2A_N$.
\end{lemma}

\begin{proof}
For $w>\tau_N$, set
\[
 h_N(w)=\int\frac{1}{(w-t)^2}\,\dd\mu_N(t).
\]
The function is continuous and strictly decreasing, since
\[
 h_N'(w)=-2\int\frac{1}{(w-t)^3}\,\dd\mu_N(t)<0.
\]
It tends to zero as $w\to\infty$.  Because $A_N/\tau_N\to\infty$,
\[
 h_N(2\tau_N)
 \ge\frac{1}{(3\tau_N)^2}
 \gg\frac1{A_N^2}.
\]
The intermediate value theorem therefore gives a unique solution
$w_N>2\tau_N$ of \eqref{eq:population-saddle}.  At this solution,
\[
 \frac1{(w_N+\tau_N)^2}
 \le h_N(w_N)=A_N^{-2}
 \le\frac1{(w_N-\tau_N)^2},
\]
and hence
\[
 w_N-\tau_N\le A_N\le w_N+\tau_N.
\]
Since $\tau_N/A_N\to0$, this directly proves $w_N/A_N\to1$.

For $w>2\tau_N$, the power series
\[
 \frac1{(w-t)^2}
 =\frac1{w^2}\sum_{\ell=0}^{\infty}(\ell+1)
 \left(\frac{t}{w}\right)^\ell
\]
converges absolutely and uniformly on the support of $\mu_N$.
Because $|t|/w\le1/2$, the terms with $\ell\ge3$ contribute at most
$Cw^{-5}\int|t|^3\,\dd\mu_N(t)$ to $h_N(w)$.  Using
\eqref{eq:truncated-moments}, we therefore have
\begin{equation}\label{eq:H-expansion}
 h_N(w)=\frac1{w^2}
 \left(
 1+\frac{2m_{1,N}}w+\frac{3m_{2,N}}{w^2}+O(w^{-3})
 \right),
\end{equation}
with an implicit constant uniform in $N$.

At $w=w_N$, the relation $h_N(w_N)=A_N^{-2}$ and $w_N\sim A_N$ give
more explicitly
\[
 \frac{w_N^2}{A_N^2}
 =1+\frac{2m_{1,N}}{w_N}
  +\frac{3m_{2,N}}{w_N^2}+O(w_N^{-3})
 =1+\frac{3\sigma^2}{A_N^2}+o(A_N^{-2}).
\]
Equivalently,
\[
 \frac{A_N^2}{w_N^2}
 =1-\frac{3\sigma^2}{A_N^2}+o(A_N^{-2}).
\]
Taking positive square roots and inverting,
\[
 w_N=A_N+\frac{3\sigma^2}{2A_N}+o(A_N^{-1}).
\]
This proves both \eqref{eq:w-asymptotic} and the first half of
\eqref{eq:wR-precise}.

Similarly,
\begin{equation}\label{eq:G-expansion}
 G_N(w)=\frac1w+\frac{m_{1,N}}{w^2}
 +\frac{m_{2,N}}{w^3}+O(w^{-4}),
\end{equation}
where the remainder comes from the terms of order $|t|^3/w^4$.  At
$w=w_N$, equations \eqref{eq:truncated-moments} and the expansion of
$w_N$ yield the component estimates
\[
 \frac1{w_N}=\frac1{A_N}-\frac{3\sigma^2}{2A_N^3}+o(A_N^{-3}),
 \qquad
 \frac{m_{1,N}}{w_N^2}=o(A_N^{-3}),
 \qquad
 \frac{m_{2,N}}{w_N^3}=\frac{\sigma^2}{A_N^3}+o(A_N^{-3}).
\]
Consequently,
\[
 G_N(w_N)
 =\frac1{A_N}-\frac{\sigma^2}{2A_N^3}+o(A_N^{-3}),
\]
and hence
\[
 A_N^2G_N(w_N)
 =A_N-\frac{\sigma^2}{2A_N}+o(A_N^{-1}).
\]
Adding the expansion of $w_N$ gives
\[
 R_N=w_N+A_N^2G_N(w_N)
 =2A_N+\frac{\sigma^2}{A_N}+o(A_N^{-1}),
\]
which implies \eqref{eq:R-asymptotic}.  As usual, polynomial factors
multiplying $e^{-c\tau_N}$ have been absorbed by decreasing the constant
$c>0$.
\end{proof}

The formula \eqref{eq:wR-precise} identifies the centering.  The
remaining arguments use $w_N\sim A_N$, the separation $w_N>2\tau_N$, and
the uniform moment bounds in \eqref{eq:truncated-moments}.


\subsection{A uniform good event}


The conditional estimates below must hold uniformly over typical
disorder configurations.  We therefore
introduce an event on which the empirical moments entering the resolvent
expansion are uniformly close to their population counterparts.

Define
\begin{equation}\label{eq:good-event}
 \cG_N=
 \bigg\{
 \max_{1\le j\le N}|\xi_j|\le\tau_N,\quad
 \max_{1\le k\le6}
 \Big|\frac1N\sum_{j=1}^N\xi_j^k-m_{k,N}\Big|
 \le N^{-1/2+\delta}
 \bigg\}.
\end{equation}
We also write $\bxi\in\cG_N$ when a deterministic configuration
satisfies both the displayed support constraint and the moment inequalities.
The range $k\le6$ is the one required by the sixth-order resolvent expansion
used in \Cref{sec:shift}.  Under $\mu_N$, the support constraint holds
almost surely.

\begin{lemma}\label{lem:good-event}
There are constants $c,C>0$ such that
\begin{equation}\label{eq:good-event-probability}
 \Prob(\cG_N^c)\le Ce^{-cN^{2\delta}}.
\end{equation}
In particular, $\cG_N^c$ is superexponentially small at every speed
appearing in \Cref{thm:main}.
\end{lemma}

\begin{proof}
Fix $1\le k\le6$ and put
\[
 Y_{j,k}=\xi_j^k-m_{k,N}.
\]
Under $\mu_N$,
\[
 |Y_{j,k}|\le2\tau_N^k,
 \qquad
 \Var(Y_{j,k})\le \E_{\mu_N}\xi_1^{2k}\le C_k,
\]
where $C_k$ is independent of $N$ by
\eqref{eq:truncated-moments}.  Let
$t_N=N^{-1/2+\delta}$.  Bernstein's inequality in Boucheron, Lugosi, and
Massart \cite[Theorem~2.10 and Corollary~2.11]{BoucheronLugosiMassart2013},
applied with $|Y_{j,k}|\le2\tau_N^k$ and
$\sum_{j=1}^N\Var(Y_{j,k})\le NC_k$, gives
\begin{align*}
 \Prob\left(
 \Big|\frac1N\sum_{j=1}^NY_{j,k}\Big|>t_N
 \right)
 \le
 2\exp\left\{-
 \frac{\frac12Nt_N^2}
 {C_k+\frac23\tau_N^k t_N}
 \right\}=
 2\exp\left\{-
 \frac{\frac12N^{2\delta}}
 {C_k+\frac23N^{k\gamma_0-1/2+\delta}}
 \right\}.
\end{align*}
Since $k\gamma_0\le6/24=1/4$ and $\delta<1/100$, the second term in the
denominator tends to zero.  A union bound over the six values of $k$
therefore yields \eqref{eq:good-event-probability}.

Finally, $a_N^{3/2}$, $a_N^3$, and $\alpha_N^2a_N^2$ are all
$N^{o(1)}$, while $N^{2\delta}$ is a fixed positive power of $N$.  Thus
$\cG_N^c$ is superexponentially small at each of these speeds.
\end{proof}

On $\cG_N$, the empirical resolvent coefficients are uniformly bounded,
and the conditional edge can be expanded about the population saddle
$w_N$.  The random part of that expansion is the edge displacement
analyzed in the next section.


\section{Gaussian moderate deviations for the edge displacement}\label{sec:shift}


This section isolates the random motion of the edge caused by the diagonal
disorder.  For a truncated configuration $\bxi$, define
\begin{equation}\label{eq:vN}
 v_N(\bxi)=w_N+S_N\sum_{j=1}^N\frac1{w_N-\xi_j}.
\end{equation}
All denominators are positive because $w_N>2\tau_N$.  Moreover, since
$N S_N=A_N^2$,
\begin{equation}\label{eq:vN-bounds}
 w_N+\frac{A_N^2}{w_N+\tau_N}
 \le v_N(\bxi)
 \le w_N+\frac{A_N^2}{w_N-\tau_N}.
\end{equation}
Thus $v_N(\bxi)\asymp w_N$ uniformly over all truncated configurations.
Here $v_N$ is the first-order random saddle parameter about $w_N$, 
whereas $R_N$ is the deterministic centering.

Define
\begin{equation}\label{eq:chi-s}
 \chi_N=\frac{\sqrt N}{\alpha_N}(\lambda_N-v_N),
 \qquad
 s_N=\frac{\sqrt N}{\alpha_N}(v_N-R_N).
\end{equation}
Then the decomposition
\begin{equation}\label{eq:Z-decomposition}
 \frac{\sqrt N}{\alpha_N}(\lambda_N-R_N)=\chi_N+s_N
\end{equation}
is exact for the truncated ensemble.  The first term will be handled by
the conditional Airy analysis, while the present section proves a
Gaussian MDP for the second term.

Set
\begin{equation}\label{eq:UN}
 U_N=\alpha_Ns_N.
\end{equation}
By \eqref{eq:RN}, \eqref{eq:vN}, and $S_N=\alpha_N^2N^{-2/3}$,
\begin{equation}\label{eq:U-exact}
 U_N
 =\sqrt N\bigl(v_N-R_N\bigr)
 =\alpha_N^2N^{-1/6}
 \sum_{j=1}^N
 \Big(
 \frac1{w_N-\xi_j}-G_N(w_N)
 \Big).
\end{equation}
The extra factor $\alpha_N$ in the definition of $U_N$ is chosen so
that, on the moderate scale $a_N$,
\[
 \frac{U_N}{\alpha_Na_N}=\frac{s_N}{a_N}.
\]
Consequently, an MDP for $U_N/(\alpha_Na_N)$ with speed
$(\alpha_Na_N)^2$ is precisely an MDP for the normalized displacement
$s_N/a_N$.

We first prove the Gaussian MDP for the leading linear statistic.

\begin{lemma}\label{lem:sum-mdp}
Let $b_N\to\infty$, $b_N=o(\sqrt N)$, and $b_N=N^{o(1)}$.  Define
\[
 T_N=\frac1{\sqrt N}
 \sum_{j=1}^N(\xi_j-m_{1,N}).
\]
Then $T_N/b_N$ satisfies an MDP on $\R$ with speed $b_N^2$ and good rate
function
\begin{equation}\label{eq:Gaussian-rate}
 I_G(x)=\frac{x^2}{2\sigma^2}.
\end{equation}
\end{lemma}

\begin{proof}
All expectations in this proof are taken under the truncated law
$\mu_N$.  Let
\[
 \Lambda_N(t)=\log\E\exp\{t(\xi_1-m_{1,N})\}.
\]
Choose $t_0\in(0,\kappa/2)$.  Since $p_N\to1$,
\[
 \sup_{N\ge N_0}\E_{\mu_N}
 \Big[(1+|\xi_1|^3)e^{t_0|\xi_1|}\Big]<\infty.
\]
Jensen's inequality gives
$\E_{\mu_N}e^{t(\xi_1-m_{1,N})}\ge1$, so the denominators appearing in
the derivatives of $\Lambda_N$ are uniformly bounded away from zero.
Hence the first three derivatives of $\Lambda_N$ are uniformly bounded
on $[-t_0/2,t_0/2]$.  Since $\Lambda_N'(0)=0$, Taylor's theorem gives
\begin{equation}\label{eq:mgf-expansion}
 \Big|\Lambda_N(t)-\frac{\sigma_N^2t^2}{2}\Big|
 \le C|t|^3,
 \qquad |t|\le t_0/2,
\end{equation}
where, by \eqref{eq:truncated-moments},
\[
 \sigma_N^2=\Var_{\mu_N}(\xi_1)
 =m_{2,N}-m_{1,N}^2
 =\sigma^2+o(1).
\]

For fixed $\lambda\in\R$, the moderate assumption $b_N=o(\sqrt N)$
places $\lambda b_N/\sqrt N$ inside the Taylor window for all large $N$,
and therefore
\begin{align*}
 \frac1{b_N^2}
 \log\E\exp\left\{\lambda b_NT_N\right\}
 &=\frac{N}{b_N^2}
 \Lambda_N\left(\frac{\lambda b_N}{\sqrt N}\right)\\
 &=\frac{\lambda^2\sigma_N^2}{2}
 +O\left(\frac{b_N}{\sqrt N}\right)
 \longrightarrow\frac{\lambda^2\sigma^2}{2}.
\end{align*}
The limiting cumulant generating function is finite and differentiable
on all of $\R$, so the essential smoothness condition in the
G\"artner--Ellis theorem is automatic.  Applying that theorem to the
triangular array $T_N/b_N$, in the form given by Dembo and
Zeitouni~\cite[Theorem~2.3.6]{DemboZeitouni1998},
gives the MDP with rate equal to the Legendre transform of
$\lambda\mapsto\lambda^2\sigma^2/2$, namely \eqref{eq:Gaussian-rate}.
\end{proof}

We next transfer this MDP from the leading linear statistic to the full
resolvent displacement in \eqref{eq:U-exact}.  This is the main
normalization step: the coefficient of $T_N$ must be shown to be
$1+o(1)$, while every higher-order resolvent term must be negligible at
speed $b_N^2$.

\begin{proposition}
\label{prop:shift-mdp}
Let $b_N\to\infty$, $b_N=o(\sqrt N)$, and $b_N=N^{o(1)}$.  Then $U_N/b_N$
satisfies an MDP with speed $b_N^2$ and rate $I_G$ in
\eqref{eq:Gaussian-rate}.  In particular, the proposition applies to
$b_N=\alpha_N a_N$.
\end{proposition}

\begin{proof}
We prove exponential equivalence between $U_N$ and the linear statistic
$T_N$ from \Cref{lem:sum-mdp}.  Since
$|\xi_j|/w_N\le\tau_N/w_N=o(1)$, the exact sixth-order expansion
\begin{equation}\label{eq:resolvent-six}
 \frac1{w_N-\xi_j}
 =\sum_{k=0}^{6}\frac{\xi_j^k}{w_N^{k+1}}
 +\frac{\xi_j^7}{w_N^7(w_N-\xi_j)}
\end{equation}
holds.  The $k=0$ term disappears after subtracting the $\mu_N$
expectation, because
\[
 G_N(w_N)=\int\frac1{w_N-t}\,\dd\mu_N(t).
\]
Set
\begin{equation}\label{eq:linear-shift-term}
 d_N=\frac{\alpha_N^2N^{-1/6}}{w_N^2}.
\end{equation}
Then the $k=1$ contribution to $U_N$ is
\[
 d_N\sum_{j=1}^N(\xi_j-m_{1,N})
 =d_N\sqrt N\,T_N.
\]
By \eqref{eq:w-asymptotic}, and in fact by the sharper expansion
\eqref{eq:wR-precise},
\begin{equation}\label{eq:dN-asymptotic}
 d_N\sqrt N=\frac{A_N^2}{w_N^2}=1+\rho_N,
 \qquad
 \rho_N=-\frac{3\sigma^2}{A_N^2}+o(A_N^{-2}),
 \qquad
 |\rho_N|\le N^{-1/3+o(1)}.
\end{equation}

For $2\le k\le6$, define
\begin{equation}\label{eq:ckN}
 c_{k,N}=\frac{\alpha_N^2N^{-1/6}}{w_N^{k+1}}
 =\alpha_N^{1-k}N^{-(k+2)/6}(1+o(1)),
\end{equation}
and let
\[
 R_{j,N}=\frac{\xi_j^7}{w_N^7(w_N-\xi_j)}.
\]
With this notation, the complete decomposition is
\begin{equation}\label{eq:Delta-decomposition}
 U_N-T_N
 =\rho_NT_N
 +\sum_{k=2}^{6}c_{k,N}\sum_{j=1}^N(\xi_j^k-m_{k,N})
 +\alpha_N^2N^{-1/6}\sum_{j=1}^N(R_{j,N}-\E R_{j,N}).
\end{equation}
We estimate the three groups of terms separately.

\medskip
\noindent{{\it The linear normalization error}.} \
Bernstein's inequality, applied to the bounded centered variables
$\xi_j-m_{1,N}$, gives
\begin{equation}\label{eq:sum-Bernstein-window}
 \Prob(|T_N|>u)
 \le 2\exp\left\{-c\min\left(u^2,
 {\tau_N^{-1}}{u\sqrt N}\right)\right\},
 \qquad u>0.
\end{equation}
If $\rho_N=0$, there is nothing to prove.  Otherwise, apply
\eqref{eq:sum-Bernstein-window} with $u=\eps b_N/|\rho_N|$.  Since
$|\rho_N|\le N^{-1/3+o(1)}$, $b_N=N^{o(1)}$, and
$\tau_N=N^{1/24}$, both
\[
 \frac{u^2}{b_N^2}=\frac{\eps^2}{\rho_N^2}
 \quad\text{and}\quad
 \frac{u\sqrt N}{\tau_N b_N^2}
 =\frac{\eps\sqrt N}{\tau_N b_N|\rho_N|}
\]
diverge to infinity.  Consequently,
\begin{equation}\label{eq:linear-term-negligible}
 \limsup_{N\to\infty}\frac1{b_N^2}
 \log\Prob(|\rho_N T_N|>\eps b_N)=-\infty.
\end{equation}

\medskip
\noindent{{\it The terms of orders $2$ through $6$}.} \
Bernstein's inequality yields
\begin{align}\label{eq:Bernstein-higher}
\Prob\bigg(
 \Big|c_{k,N}\sum_{j=1}^N
 (\xi_j^k-m_{k,N})\Big|>\eps b_N
 \bigg)\le2\exp\bigg\{-c\min\bigg(
 \frac{\eps^2b_N^2}{Nc_{k,N}^2},
 \frac{\eps b_N}{|c_{k,N}|\tau_N^k}
 \bigg)\bigg\}.
\end{align}
After division by $b_N^2$, the first exponent contains
\[
 \frac1{Nc_{k,N}^2}
 =\bigl(\alpha_N^2N^{1/3}\bigr)^{k-1}(1+o(1))\to\infty,
\]
where the divergence follows from the subpolynomial assumption on
$\alpha_N^{-1}$.  The second exponent contains
\[
 \frac1{b_N|c_{k,N}|\tau_N^k}
 =\frac{\alpha_N^{k-1}}{b_N}
   N^{(k+2)/6-k/24}(1+o(1))\longrightarrow\infty,
\]
because both $b_N$ and $\alpha_N^{\pm1}$ are subpolynomial.  Hence every
term with $2\le k\le6$ is superexponentially negligible at speed
$b_N^2$.

\medskip
\noindent{{\it A deterministic bound for the seventh-order remainder.}} \
Since $w_N-\tau_N\ge w_N/2$ for large $N$,
\[
 |R_{j,N}|
 \le C\frac{\tau_N^7}{w_N^8},
 \qquad
 |\E R_{j,N}|
 \le C\frac{\tau_N^7}{w_N^8}.
\]
Therefore
\begin{align*}
&\Big|\alpha_N^2N^{-1/6}\sum_{j=1}^N
 (R_{j,N}-\E R_{j,N})\Big|\\
&\le C\alpha_N^{-6}N^{-1/2}\tau_N^7(1+o(1))
=N^{-5/24+o(1)}=o(b_N).
\end{align*}
In particular, for every fixed $\eps>0$, this remainder term is
eventually smaller than $\eps b_N/7$.

Combining the three estimates in \eqref{eq:Delta-decomposition}, and
using a union bound with $\eps/7$ for each of the six random terms,
yields, for every $\eps>0$,
\begin{equation}\label{eq:shift-exp-equivalence}
 \limsup_{N\to\infty}\frac1{b_N^2}
 \log\Prob(|U_N-T_N|>\eps b_N)=-\infty.
\end{equation}
Thus $U_N/b_N$ and $T_N/b_N$ are exponentially equivalent at speed
$b_N^2$.  Exponential equivalence can transfer the MDP from $T_N/b_N$ to
$U_N/b_N$; see Dembo and
Zeitouni~\cite[Theorem~4.2.13]{DemboZeitouni1998}.
\end{proof}

For the choice $b_N=\alpha_Na_N$, \Cref{prop:shift-mdp} says that
\[
 \frac{s_N}{a_N}=\frac{U_N}{\alpha_Na_N}
\]
satisfies a Gaussian MDP with speed $\alpha_N^2a_N^2$ and rate
$x^2/(2\sigma^2)$.  This is exactly the displacement estimate used in the
conditional convolution argument of \Cref{sec:phase-proof}.


\section{Conditional kernel and saddle geometry}\label{sec:kernel}


We now turn to the conditional fluctuation of the largest eigenvalue
about the random saddle parameter $v_N$.  This section has three results:
the conditional Fredholm determinant, a uniform cubic expansion of the
phase near $w_N$, and a global decay estimate away from the saddle.  The
last two are exactly the estimates needed for the trace-norm Airy
approximation in \Cref{sec:airy}.

Let $\cF_N=\sigma(\xi_1,\ldots,\xi_N)$.  Conditionally on $\cF_N$, the
eigenvalues of the truncated matrix form a determinantal point process.
When no ambiguity can arise, we suppress the configuration argument
$\bxi$ from $v_N$, $f_N$, $\Phi_N$, and the kernels.  Johansson's conditional
kernel~\cite[Section~5, Equations~(67)--(70)]{Johansson2007} is
\begin{equation}\label{eq:Johansson-kernel}
 K_N(u,v;\bxi)
 =\frac1{(2\pi {\rm i})^2S_N}
 \int_\gamma\dd z\int_\Gamma\dd w\,
 e^{(w-v)^2/(2S_N)-(z-u)^2/(2S_N)}
 \frac1{w-z}
 \prod_{j=1}^N\frac{w-\xi_j}{z-\xi_j}.
\end{equation}
Here $\gamma$ is a positively oriented simple contour enclosing all
$\xi_j$, and $\Gamma$ is an upward-oriented vertical contour to the
right of $\gamma$.  For coinciding diagonal entries, the coalescing
limit of the distinct-entry formula may be taken with fixed common contours,
because the integrand and the conditional matrix law are continuous
in the diagonal entries.

The inclusion--exclusion formula for a finite determinantal point
process (see Hough, Krishnapur, Peres, and
Vir\'ag~\cite[Chapter~4]{HoughKrishnapurPeresVirag2009}) gives the gap
probability as the finite Fredholm series below; at this stage, we use that
finite series to define the determinant in
\begin{equation}\label{eq:conditional-Fredholm}
 \Prob(\lambda_N\le t\mid\cF_N)
 =\det(I-K_N)_{L^2(t,\infty)}.
\end{equation}
Indeed, since the conditional process has exactly $N$ particles, its
correlation functions, and hence its Fredholm minors, vanish for
$k>N$.  Thus the finite series is
\begin{equation}\label{eq:Fredholm-series}
 \det(I-K_N)_{L^2(t,\infty)}
 =\sum_{k=0}^N\frac{(-1)^k}{k!}
 \int_{(t,\infty)^k}
 \det\bigl(K_N(x_i,x_j;\bxi)\bigr)_{i,j=1}^k
 \dd x_1\cdots\dd x_k.
\end{equation}
This finite series interpretation also justifies the conjugation in
\Cref{sec:airy} term by term.

Next, we introduce the phase locally by
\begin{equation}\label{eq:phase}
 f_N(w;\bxi)
 =\frac{w^2}{2S_N}-\frac{v_N(\bxi)w}{S_N}
 +\sum_{j=1}^N\log(w-\xi_j).
\end{equation}
On the disk $\{w:|w-w_N|<w_N-\tau_N\}$, use the analytic branches
normalized by $\log(w_N-\xi_j)\in\R$; along each open deformed ray,
continue these branches from the saddle.  On the original closed
$z$-contour, we use the single-valued identity
\begin{equation}\label{eq:phase-difference-product}
 e^{f_N(w)-f_N(z)}
 =\exp\bigg\{\frac{w^2-z^2}{2S_N}
 -\frac{v_N(\bxi)(w-z)}{S_N}\bigg\}
 \prod_{j=1}^N\frac{w-\xi_j}{z-\xi_j}.
\end{equation}
Consequently, the derivatives and exponentiated phase differences used
below are well defined independently of the local branch choices.

By the definition of $v_N$ in \eqref{eq:vN},
\begin{equation}\label{eq:fprime-zero}
 f_N'(w_N;\bxi)=0.
\end{equation}
Moreover,
\begin{align}
 f_N''(w_N;\bxi)
 &=\frac1{S_N}-\sum_{j=1}^N\frac1{(w_N-\xi_j)^2},
 \label{eq:fsecond}\\
 f_N^{(3)}(w_N;\bxi)
 &=2\sum_{j=1}^N\frac1{(w_N-\xi_j)^3},
 \label{eq:fthird}\\
 f_N^{(m)}(w_N;\bxi)
 &=(-1)^{m-1}(m-1)!
 \sum_{j=1}^N\frac1{(w_N-\xi_j)^m},
 \qquad m\ge3.                                         \label{eq:fhigher}
\end{align}
The population saddle equation makes the quadratic coefficient centered:
\begin{equation}\label{eq:fsecond-expectation}
 \E f_N''(w_N)
 =\frac1{S_N}-N\int\frac1{(w_N-t)^2}\,\dd\mu_N(t)
 =\frac{N}{A_N^2}-\frac{N}{A_N^2}=0.
\end{equation}
The natural local scale is
\begin{equation}\label{eq:ellN}
 \ell_N=\alpha_NN^{-1/6},
\end{equation}
for which
\[
 \frac{w_N}{\ell_N}=N^{1/3}(1+o(1)),
 \qquad
 \frac{N\ell_N^3}{w_N^3}=1+o(1).
\]
This is the scale on which the third derivative contributes a unit
Airy cubic.


\subsection{Uniform local cubic expansion}


Define
\begin{equation}\label{eq:scaled-phase}
 \Phi_N(r;\bxi)
 =f_N(w_N+\ell_Nr;\bxi)-f_N(w_N;\bxi).
\end{equation}
The next lemma identifies the Airy phase uniformly over typical disorder
configurations.  The derivative estimates are included because they will
be used when the unshifted rays are displaced by $d\ell_N$.

\begin{lemma}\label{lem:cubic-expansion}
There are constants $c_0,c,C>0$ such that, on $\cG_N$ and uniformly for
$|r|\le N^{c_0}$,
\begin{equation}\label{eq:cubic-error}
 \bigg|\Phi_N(r;\bxi)-\frac{r^3}{3}\bigg|
 \le CN^{-c}(1+|r|^4).
\end{equation}
The first two derivatives satisfy
\begin{align}
 \left|\Phi_N'(r;\bxi)-r^2\right|
 &\le CN^{-c}(1+|r|^3),                                 \label{eq:cubic-error-1}\\
 \left|\Phi_N''(r;\bxi)-2r\right|
 &\le CN^{-c}(1+|r|^2).                                 \label{eq:cubic-error-2}
\end{align}
\end{lemma}

\begin{proof}
Let $h=\ell_Nr$.  Since
\[
 \frac{|h|}{w_N-\tau_N}
 \le N^{c_0-1/3+o(1)},
\]
we may choose $c_0<1/12$ so that all logarithmic Taylor series below
are absolutely and uniformly convergent.

We first control the quadratic coefficient.  By
\eqref{eq:fsecond-expectation},
\begin{equation}\label{eq:fsecond-centered}
 f_N''(w_N)
 =N\int\frac1{(w_N-t)^2}\,\dd\mu_N(t)
 -\sum_{j=1}^N\frac1{(w_N-\xi_j)^2}.
\end{equation}
For $|t|\le\tau_N$, expand to order six:
\[
 \frac1{(w_N-t)^2}
 =\sum_{k=0}^{6}\frac{(k+1)t^k}{w_N^{k+2}}
 +O\left(\frac{|t|^7}{w_N^9}\right).
\]
The $k=0$ terms cancel in \eqref{eq:fsecond-centered}.  On $\cG_N$,
the empirical moment errors for $1\le k\le6$ give
\begin{equation}\label{eq:fsecond-bound}
 |f_N''(w_N)|
 \le C\frac{N^{1/2+\delta}}{w_N^3}
 +CN\frac{\tau_N^7}{w_N^9}.
\end{equation}
Multiplication by $\ell_N^2$ and the relation
$w_N=\alpha_NN^{1/6}(1+o(1))$ give
\begin{equation}\label{eq:quadratic-scaled-bound}
 |f_N''(w_N)|\ell_N^2
 \le N^{-1/3+\delta+o(1)}+N^{-13/24+o(1)}.
\end{equation}

Next, by \eqref{eq:fthird},
\begin{equation}\label{eq:cubic-coefficient}
 \frac{\ell_N^3}{6}f_N^{(3)}(w_N)
 =\frac{\alpha_N^3N^{-1/2}}3
 \sum_{j=1}^N\frac1{(w_N-\xi_j)^3}.
\end{equation}
Expanding $(w_N-t)^{-3}$ to order six and using $\cG_N$ yields
\begin{align*}
 \frac1N\sum_{j=1}^N\frac1{(w_N-\xi_j)^3}
 &=\int\frac1{(w_N-t)^3}\,\dd\mu_N(t)\\
 &\quad+O(N^{-1/2+\delta}w_N^{-4})
 +O(\tau_N^7w_N^{-10}).
\end{align*}
By \eqref{eq:truncated-moments},
\[
 \int\frac1{(w_N-t)^3}\,\dd\mu_N(t)
 =w_N^{-3}\left(1+O(w_N^{-2})+O(e^{-c\tau_N})\right).
\]
Together with \eqref{eq:w-asymptotic}, this gives
\begin{equation}\label{eq:cubic-coefficient-limit}
 \frac{\ell_N^3}{6}f_N^{(3)}(w_N)
 =\frac13+O(N^{-c})
\end{equation}
for some $c>0$.

It remains to bound the terms of order four and higher.  The distance
from $w_N$ to every pole is at least $w_N-\tau_N\asymp w_N$.
Hence, for $|h|<w_N-\tau_N$,
\begin{align}\label{eq:fourth-remainder}
 \bigg|
 \sum_{m=4}^{\infty}\frac{f_N^{(m)}(w_N)}{m!}h^m
 \bigg|
 \le
 CN\bigg(\frac{|h|}{w_N}\bigg)^4
 \frac1{1-|h|/(w_N-\tau_N)}\le CN^{-1/3}|r|^4;
\end{align}
Here we used $h=\ell_Nr$ in the last inequality. Let $\mathcal R_N(r)$ denote the series on the left with
$h=\ell_Nr$.  For $q=0,1,2$, termwise differentiation and the same
geometric series estimate give
\begin{equation}\label{eq:remainder-derivatives}
 |\mathcal R_N^{(q)}(r)|
 \le
 CN\bigg(\frac{\ell_N}{w_N}\bigg)^4
 \frac{1+|r|^{4-q}}
 {\left(1-\ell_N|r|/(w_N-\tau_N)\right)^{q+1}}
 \le CN^{-1/3}(1+|r|^{4-q}).
\end{equation}
The quadratic estimate \eqref{eq:quadratic-scaled-bound}, the cubic
estimate \eqref{eq:cubic-coefficient-limit}, and
\eqref{eq:fourth-remainder} prove \eqref{eq:cubic-error} after
decreasing $c_0$ and $c$ if necessary.  Differentiating the quadratic
and cubic terms and using \eqref{eq:remainder-derivatives} proves
\eqref{eq:cubic-error-1} and \eqref{eq:cubic-error-2}.
\end{proof}

The local expansion identifies the Airy scale.  For uniformity over
truncated configurations, the monotone decay required for the global contour
deformation is derived directly from the deterministic support bound.


\subsection{Unshifted steepest-descent rays}


For $t\ge0$, let
\begin{align*}
 C_1(t)&=w_N+t+{\rm i}t,& C_2(t)&=w_N+t-{\rm i}t,\\
 C_3(t)&=w_N-t+{\rm i}t,& C_4(t)&=w_N-t-{\rm i}t.
\end{align*}
The $w$-contour will ultimately be deformed to $C_1-C_2$, and the
$z$-contour to $C_3-C_4$.  The following exact derivative identities
give the required monotone decay.

\begin{lemma}\label{lem:global-decay}
Let $a_j=w_N-\xi_j$.  For $j=1,2$, define
\[
 g_j(t)=\Re\{f_N(C_j(t))-f_N(w_N)\},
\]
and for $j=3,4$ define
\[
 g_j(t)=-\Re\{f_N(C_j(t))-f_N(w_N)\}.
\]
Then
\begin{align}
 g_1'(t)=g_2'(t)
 &=-\sum_{k=1}^N
 \frac{2t^2}{a_k\{(a_k+t)^2+t^2\}},                    \label{eq:gplus-prime}\\
 g_3'(t)=g_4'(t)
 &=-\sum_{k=1}^N
 \frac{2t^2}{a_k\{(a_k-t)^2+t^2\}}.                    \label{eq:gminus-prime}
\end{align}
Consequently, uniformly over all truncated configurations (and hence on
$\cG_N$), there is a constant $c>0$ such that
\begin{align}
 g_j(t)&\le-c\frac{Nt^3}{w_N^3},
 &&0\le t\le\frac{w_N}{2},                             \label{eq:unshifted-local-decay}\\
 g_j(t)&\le-cN-c\frac{N}{w_N}\Big(t-\frac{w_N}{2}\Big),
 &&\frac{w_N}{2}\le t<\infty.                                  \label{eq:unshifted-far-decay}
\end{align}
\end{lemma}

\begin{proof}
We prove the formulas on the upper rays; the lower rays are their
complex conjugates.  For $C_1$, differentiating the real part of
\eqref{eq:phase} gives
\[
 g_1'(t)=\frac{w_N-v_N}{S_N}
 +\sum_{k=1}^N
 \frac{a_k+2t}{(a_k+t)^2+t^2}.
\]
At $t=0$, this vanishes by \eqref{eq:fprime-zero}.  Subtracting the
value at zero term by term,
\[
 \frac{a_k+2t}{(a_k+t)^2+t^2}-\frac1{a_k}
 =-\frac{2t^2}{a_k\{(a_k+t)^2+t^2\}},
\]
which proves \eqref{eq:gplus-prime}.

For $C_3$, differentiation first gives
\[
 \frac{\dd}{\dd t}\Re f_N(C_3(t))
 =-\frac{w_N-v_N}{S_N}
 +\sum_{k=1}^N
 \frac{-a_k+2t}{(a_k-t)^2+t^2}.
\]
Taking the negative and again using \eqref{eq:fprime-zero} gives
\[
 -\frac1{a_k}+\frac{a_k-2t}{(a_k-t)^2+t^2}
 =-\frac{2t^2}{a_k\{(a_k-t)^2+t^2\}},
\]
which proves \eqref{eq:gminus-prime}.  In particular, the denominator
on the left rays contains $(a_k-t)^2$, so no sign change occurs when
the ray passes to the left of a pole.

Since $|\xi_k|\le\tau_N=o(w_N)$, one has $a_k\asymp w_N$ uniformly
over all truncated configurations.  If $0\le t\le w_N/2$, every
denominator in \eqref{eq:gplus-prime}--\eqref{eq:gminus-prime} is at
most $Cw_N^3$.  Thus
\[
 g_j'(t)\le-c\frac{Nt^2}{w_N^3},
\]
and integration from zero proves \eqref{eq:unshifted-local-decay}.
If $t\ge w_N/2$, then
\[
 (a_k\pm t)^2+t^2\le C t^2,
\]
so each summand in absolute value is at least $c/w_N$.  Integrating
from $w_N/2$, where \eqref{eq:unshifted-local-decay} already gives a
$-cN$ bound, proves \eqref{eq:unshifted-far-decay}.
\end{proof}

\begin{remark}\label{rem:unshifted-rays}
The rays $C_1,\ldots,C_4$ are used only to establish monotonicity and
decay of the phase.  The actual integration contours will be displaced
by $d\ell_N$ to keep the Cauchy factor $(w-z)^{-1}$ uniformly
nonsingular.  The complete deformation, including the closing segments
at infinity, is proved in
\Cref{prop:scaled-contour-representation} after the corresponding
decay estimates have been established.
\end{remark}

With the conditional determinant, the local Airy phase, and the global
ray decay in place, we can now compare the finite-$N$ kernel with the
Airy kernel in trace norm.


\section{Airy approximation on a logarithmic window}\label{sec:airy}


We strengthen Johansson's compact-window asymptotics~\cite{Johansson2007} to a
trace-norm estimate on the logarithmic window required for moderate
deviations.  The argument has four steps.  We begin by rescaling and conjugating the kernel, which
leaves every Fredholm minor unchanged. Next, we separate the two contours so that the Cauchy factor
admits a convergent Laplace representation. The resulting kernel is then written as a product of two
Hilbert--Schmidt operators. Finally, the finite-$N$ factors are compared
with their Airy limits. This factorization  reduces the kernel estimate to four tractable
Hilbert--Schmidt estimates.


\subsection{Rescaling and separated contours}


The first step is purely algebraic.  We express the spectral argument on the
Airy scale about $v_N$ and choose a conjugation that removes all constant
terms in the two spatial variables.  We  then apply this conjugation term by
term in the finite Fredholm series \eqref{eq:Fredholm-series}.

For $x,y\in\R$, set
\[
 u_x=v_N+\frac{\alpha_Nx}{\sqrt N},
 \qquad
 u_y=v_N+\frac{\alpha_Ny}{\sqrt N}.
\]
Define
\begin{equation}\label{eq:QN}
 Q_N(x)=\frac{u_x^2}{2S_N}
 -\frac{N^{1/6}w_N}{\alpha_N}x
\end{equation}
and the rescaled, conjugated kernel
\begin{equation}\label{eq:conjugated-kernel}
 \widetilde K_N(x,y;\bxi)
 =\frac{\alpha_N}{\sqrt N}
 e^{Q_N(x)-Q_N(y)}K_N(u_x,u_y;\bxi).
\end{equation}
For every $k\le N$, cancellation of the row and column factors gives
\[
 \det\bigl(\widetilde K_N(x_i,x_j)\bigr)_{i,j=1}^k
 =\left(\frac{\alpha_N}{\sqrt N}\right)^k
 \det\bigl(K_N(u_{x_i},u_{x_j})\bigr)_{i,j=1}^k,
\]
because the product of the row factors
$e^{Q_N(x_i)}$ cancels the product of the column factors
$e^{-Q_N(x_j)}$.  The change of variables
$u_{x_i}=v_N+\alpha_Nx_i/\sqrt N$ supplies the corresponding Jacobian factor in the
$k$-fold integral.  Hence the two Fredholm series agree term by term.
In particular, a trace-norm comparison for $\widetilde K_N$ gives a
comparison for the original conditional Fredholm determinant.

We next choose the contour geometry.  The $w$-contour is shifted to the
right of the saddle and the $z$-contour to the left.  This separation is
what will later make $(w-z)^{-1}$ uniformly nonsingular and allow the
Cauchy factor to be written as a Laplace integral.

Fix $d\in(0,1/10)$.  In the scaled variables
$w=w_N+\ell_Nr$ and $z=w_N+\ell_Ns$, use the separated contours
\begin{equation}\label{eq:shifted-contours-def}
\begin{aligned}
 \Gamma^\pm&=\{d+t(1\pm {\rm i}):t\ge0\},\\
 \Sigma^\pm&=\{-d-t\pm {\rm i}t:t\ge0\},
\end{aligned}
\end{equation}
each ray being parametrized from its finite endpoint towards infinity.  Put
$\Gamma=\Gamma^+-\Gamma^-$ and $\Sigma=\Sigma^+-\Sigma^-$.  Thus both
oriented contours run from their lower infinite ray through the finite
endpoint to their upper infinite ray.  The geometry and orientation are shown in \Cref{fig:contours}.  Then
\begin{equation}\label{eq:contour-separation}
 \Re(r-s)\ge2d,
 \qquad r\in\Gamma,
 \quad s\in\Sigma.
\end{equation}

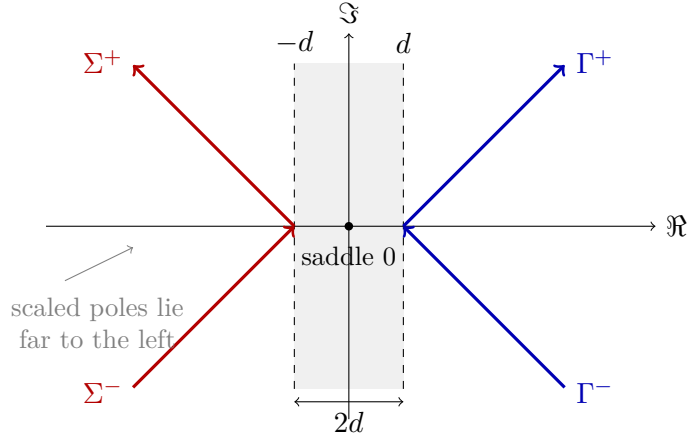
\begin{figure}[t]
\centering
\begin{tikzpicture}[x=1cm,y=1cm]
  \fill[gray!12] (-0.72,-2.15) rectangle (0.72,2.15);
  \draw[->] (-4.0,0) -- (4.05,0) node[right] {$\Re$};
  \draw[->] (0,-2.55) -- (0,2.55) node[above] {$\Im$};
  \draw[dashed] (-0.72,-2.15) -- (-0.72,2.15);
  \draw[dashed] (0.72,-2.15) -- (0.72,2.15);
  \node[above] at (-0.72,2.15) {$-d$};
  \node[above] at (0.72,2.15) {$d$};
  \draw[<->] (-0.72,-2.32) -- (0.72,-2.32);
  \node[below] at (0,-2.32) {$2d$};

  \fill (0,0) circle (1.6pt);
  \node[anchor=north,font=\small] at (0,-0.14) {saddle $0$};

  \draw[very thick,blue!70!black,->] (0.72,0) -- (2.85,2.13);
  \draw[very thick,blue!70!black,->] (2.85,-2.13) -- (0.72,0);
  \node[blue!70!black] at (3.25,2.18) {$\Gamma^+$};
  \node[blue!70!black] at (3.25,-2.18) {$\Gamma^-$};

  \draw[very thick,red!70!black,->] (-0.72,0) -- (-2.85,2.13);
  \draw[very thick,red!70!black,->] (-2.85,-2.13) -- (-0.72,0);
  \node[red!70!black] at (-3.25,2.18) {$\Sigma^+$};
  \node[red!70!black] at (-3.25,-2.18) {$\Sigma^-$};

  \draw[gray,->] (-3.75,-0.72) -- (-2.85,-0.28);
  \node[gray,align=center,font=\small] at (-3.32,-1.28)
    {scaled poles lie\\far to the left};
\end{tikzpicture}
\caption{\scriptsize Separated Airy contours in the scaled complex plane.  Both contours are oriented from the lower ray to the upper ray.  The $r$-contour $\Gamma$ lies to the right of the saddle and the $s$-contour $\Sigma$ lies to the left.  Their separation gives $\Re(r-s)\ge2d$, which makes the Laplace representation of the Cauchy factor absolutely convergent.  The scaled poles $(\xi_j-w_N)/\ell_N$ remain at distance $N^{1/3+o(1)}$ to the left.}
\label{fig:contours}
\end{figure}
To verify the rescaling, write $u_x=v_N+\alpha_Nx/\sqrt N$ and
$u_y=v_N+\alpha_Ny/\sqrt N$.  The exponent in
\eqref{eq:Johansson-kernel} is exactly
\begin{align*}
 &f_N(w)-f_N(z)
 -\frac{N^{1/6}}{\alpha_N}yw
 +\frac{N^{1/6}}{\alpha_N}xz
 +\frac{u_y^2-u_x^2}{2S_N}.
\end{align*}
After $w=w_N+\ell_Nr$ and $z=w_N+\ell_Ns$, the constant terms in $x,y$
are cancelled by $Q_N(x)-Q_N(y)$.  In addition, note that
\[
 \frac{\alpha_N}{\sqrt N}\frac1{S_N}
 \frac{\ell_N^2}{\ell_N}=1,
\]
thus no additional scalar prefactor remains.

This algebra yields the desired finite-$N$ double-contour representation
on $\Gamma$ and $\Sigma$. Specifically, for all $x,y\in\R$,
\begin{equation}\label{eq:scaled-double-contour}
 \widetilde K_N(x,y;\bxi)
 =\frac1{(2\pi {\rm i})^2}
 \int_\Sigma\dd s\int_\Gamma\dd r\,
 \frac{e^{\Phi_N(r)-yr-\Phi_N(s)+xs}}{r-s},
\end{equation}
where $\Phi_N(r)$ is defined by \eqref{eq:scaled-phase}. The deformation from Johansson's original
contours $(\gamma,\Gamma)$ in \eqref{eq:Johansson-kernel} to the separated rays $(\Sigma,\Gamma)$ still has to be justified;
we postpone that analytic step to \Cref{prop:scaled-contour-representation}, after the required global
decay has been proved.  The target Airy kernel has the parallel representation
\begin{equation}\label{eq:Airy-double-contour}
 K_{\Ai}(x,y)
 =\frac1{(2\pi {\rm i})^2}
 \int_\Sigma\dd s\int_\Gamma\dd r\,
 \frac{e^{r^3/3-yr-s^3/3+xs}}{r-s}.
\end{equation}
For completeness, insert the standard contour formulas in Olver et
al.~\cite[Section~9.5]{OlverEtAl2010}
\[
 \Ai(y+\lambda)=\frac1{2\pi {\rm i}}\int_\Gamma
 e^{r^3/3-(y+\lambda)r}\,\dd r,
 \qquad
 \Ai(x+\lambda)=\frac1{2\pi {\rm i}}\int_\Sigma
 e^{-s^3/3+(x+\lambda)s}\,\dd s.
\]
Since the Airy integrands are entire and decay on the closing arcs
inside the relevant Stokes sectors, Cauchy's theorem deforms the standard
Airy contours to the shifted rays $\Gamma$ and $\Sigma$.
The second formula is obtained from the first by the substitution $r=-s$;
the minus sign in $\dd r=-\dd s$ reverses the image contour and produces
the orientation assigned to $\Sigma$.  Because $\Re(r-s)\ge2d$, Fubini's
theorem and
$\int_0^\infty e^{-\lambda(r-s)}\dd\lambda=(r-s)^{-1}$ give
\eqref{eq:Airy-double-contour} from
$K_{\Ai}(x,y)=\int_0^\infty\Ai(x+\lambda)\Ai(y+\lambda)\dd\lambda$.


\subsection{Decay on the separated contours}


This subsection constitutes the analytical core of our approach.
On the local contour scale we need more than decay:  a quantitative comparison
between $e^{\pm\Phi_N}$ and $e^{\pm\zeta^3/3}$.  On the mesoscopic and
far scales, decay alone is enough, provided it is uniform on $\cG_N$ and
strong enough to absorb spatial variables in the logarithmic window.

Let
\begin{equation}\label{eq:TNstar}
 T_N^*=\frac{w_N}{2\ell_N}=\frac12N^{1/3}(1+o(1)).
\end{equation}
The finite-$N$ phase has cubic decay only up to this scale.  Beyond it, the
correct global estimate is linear in the contour parameter.  The next lemma
transfers the unshifted decay to the separated contours.

\begin{lemma}
\label{lem:shifted-contours}
There exist constants $c,C,\eta>0$ such that, uniformly for
$\bxi\in\cG_N$, the following statements hold for all sufficiently large
$N$.

If $r=d+t(1\pm {\rm i})\in\Gamma^\pm$, then
\begin{align}
 \Re\Phi_N(r)
 &\le C-ct^3,
 &&0\le t\le T_N^*,                                    \label{eq:shifted-r-local}\\
 \Re\Phi_N(r)
 &\le-cN-cN^{2/3}(t-T_N^*),
 &&T_N^*\le t<\infty.                                          \label{eq:shifted-r-far}
\end{align}

If $s=-d-t\pm {\rm i}t\in\Sigma^\pm$, then
\begin{align}
 -\Re\Phi_N(s)
 &\le C-ct^3,
 &&0\le t\le T_N^*,                                    \label{eq:shifted-s-local}\\
 -\Re\Phi_N(s)
 &\le-cN-cN^{2/3}(t-T_N^*),
 &&T_N^*\le t<\infty.                                          \label{eq:shifted-s-far}
\end{align}

On the local pieces $0\le t\le N^\eta$,
\begin{align}
 \Big|e^{\Phi_N(r)}-e^{r^3/3}\Big|
 \le CN^{-c}(1+t^4)e^{-ct^3},                          \label{eq:shifted-r-comparison}\\
 \Big|e^{-\Phi_N(s)}-e^{-s^3/3}\Big|
 \le CN^{-c}(1+t^4)e^{-ct^3}.                          \label{eq:shifted-s-comparison}
\end{align}

Moreover, after decreasing $\eta$ if necessary, for every fixed integer
$m\ge0$ there is a constant $C_m>0$ such that, whenever
$L^3\le c\log N$, uniformly for $0\le q\le L$,
$\rho\ge0$, and either choice of sign,
\begin{align}
 &\int_{N^\eta}^{\infty}(1+t^m)
 \exp\!\left\{\Re\Phi_N(d+t(1\pm {\rm i}))+q(d+t)-\rho(d+t)\right\}\dd t
 \le C_m e^{-cN^{3\eta}}e^{-d\rho},  \label{eq:r-tail-integral}\\
 &\int_{N^\eta}^{\infty}(1+t^m)
 \exp\!\left\{-\Re\Phi_N(-d-t\pm {\rm i}t)+q(d+t)-\rho(d+t)\right\}\dd t
 \le C_m e^{-cN^{3\eta}}e^{-d\rho}. \label{eq:s-tail-integral}
\end{align}

The same estimates hold with $\Phi_N(\zeta)$ replaced by $\zeta^3/3$.  For a
nonnegative spatial variable (denoted by $x$ on $\Sigma$ and by $y$ on
$\Gamma$), the corresponding estimates hold after extracting $e^{-dx}$ or
$e^{-dy}$, respectively, and omitting the positive term $q(d+t)$.
\end{lemma}

\begin{proof}[Sketch of proof]
The argument proceeds in four steps, treating the local, mesoscopic, and far
ranges of the contour parameter separately, and then assembling the uniform
tail integrals. For readability, we present only an outline of the proof here; see \Cref{app:shifted-contours}
for the full details.

\smallskip
\noindent{\it Step~1. Local window: $0\le t\le N^{\eta}$.}\
On this range the cubic expansion of \Cref{lem:cubic-expansion} is valid.
The exact identity $\Re((d+t+{\rm i}t)^3/3)=d^3/3+d^2t-2t^3/3$ yields the cubic decay
\eqref{eq:shifted-r-local}, while the uniform error bound
$|\Phi_N(r)-r^3/3|\le CN^{-c}(1+t^4)$ gives the pointwise comparison
\eqref{eq:shifted-r-comparison}.  The same computation on $\Sigma$ produces
\eqref{eq:shifted-s-local} and \eqref{eq:shifted-s-comparison}.

\smallskip
\noindent{\it Step~2. Mesoscopic range: $N^{\eta}\le t\le T_N^*$.}\
Here the cubic expansion is no longer accurate, but the unshifted rays
$C_1,\dots,C_4$ already carry strong decay by \Cref{lem:global-decay}.
The horizontal displacement by $d\ell_N$ is controlled by Taylor's formula
for $f_N'$:
\[
 \ell_N f_N'(w_N+\ell_N(\zeta+u))
 =\ell_N^2f_N''(w_N)(\zeta+u)
 +\ell_N^3(\zeta+u)^2\!\int_0^1\!(1-v)f_N^{(3)}(\cdot)\,\dd v,
\]
where $|\ell_N^2f_N''(w_N)|\le CN^{-c}$ and $|f_N^{(3)}|\le CN/w_N^3$.
Integrating horizontally from $r_0(t)$ to $r(t)=r_0(t)+d$ changes the real
part by at most $C(1+t^2)$, which is absorbed by the unshifted cubic decay
$\Re\Phi_N(r_0(t))\le-ct^3$ because $t\ge N^{\eta}$.

\smallskip
\noindent{\it Step~3. Far range: $t\ge T_N^*$.}\
Beyond the cubic decay scale, the unshifted phase has switched to linear
decay \eqref{eq:unshifted-far-decay}.  A direct estimate from
\eqref{eq:phase} shows that $|\ell_N f_N'|\le C(N^{2/3}+N^{1/3}t)$ along the
horizontal segment, so the displacement changes the real part by at most
$CN^{2/3}+CN^{1/3}h$ (with $t=T_N^*+h$).  Both terms are absorbed by the
terms $-cN-cN^{2/3}h$ in \eqref{eq:unshifted-far-decay}.

\smallskip
\noindent{\it Step~4. Uniform tail integrals.}\
The integrals \eqref{eq:r-tail-integral}--\eqref{eq:s-tail-integral} are
treated by splitting at $T_N^*$.  On $[N^{\eta},T_N^*]$ the cubic decay
$-ct^3$ dominates the spatial factor $q(d+t)$ because $L=o(N^{2\eta})$, whereas
on $[T_N^*,\infty)$ the linear decay $-cN-cN^{2/3}h$ dominates $L(d+T_N^*+h)$
because $LT_N^*=o(N)$ and $Lh=o(N^{2/3}h)$.  The Airy counterparts are simpler
since cubic decay holds for all $t$.
\end{proof}

The above estimates are now strong enough to justify the exact contour
representation on the separated paths.  The proof proceeds by two
independent truncations: first the closed $z$-contour is opened into two
rays, then the vertical $w$-contour is deformed into the two right rays.
Separating the two limits makes it clear that neither homotopy crosses the
Cauchy pole $w=z$.

\begin{proposition}\label{prop:scaled-contour-representation}
Fix $d\in(0,1/10)$ and let $L_N\ge0$ satisfy $L_N=o(N^{2/3})$.  For all sufficiently
large $N$, every $\bxi\in\cG_N$, and all $x,y\in[-L_N,\infty)$, the
rescaled kernel $\widetilde{K}_N$ defined by \eqref{eq:conjugated-kernel} has the exact representation \eqref{eq:scaled-double-contour}.
For each fixed $x,y$ the finite closing homotopies below are justified
pointwise.  The final ray integrals are absolutely convergent after the
factorization in \eqref{eq:Cauchy-factor}, and their estimates are uniform
for $x,y\in[-L_N,\infty)$ in the ranges used later.
\end{proposition}

\begin{proof}
\noindent{{\it Deforming the $z$-contour}.} \
We first work in the original variables.  Since
$w_N-d\ell_N>\tau_N$ for all sufficiently large $N$, a positively oriented
simple contour enclosing every $\xi_j$ may be chosen inside
$\{\Re z<w_N-d\ell_N\}$.  The vertical $w$-contour in
\eqref{eq:Johansson-kernel} may be chosen on the line
$\Re w=w_N+2d\ell_N$.  We use independent truncation parameters $B_w$ and
$B_z$.  First truncate the vertical $w$-contour at imaginary parts
$\pm B_w$.  For $B_z>w_N$, deform the closed $z$-contour to the two
truncated rays
\[
 w_N-d\ell_N-t\pm {\rm i}t,
 \qquad 0\le t\le w_N+B_z-d\ell_N,
\]
joined by the vertical segment on $\Re z=-B_z$.  This closed contour is
positively oriented and encloses every $\xi_j$.  During the deformation its
real part stays at least $d\ell_N$ to the left of the $w$-contour, so the pole
$w=z$ is not crossed.  Cauchy's theorem therefore justifies the finite
$z$-deformation.

The bounds \eqref{eq:vN-bounds}, together with $w_N/A_N\to1$ and
$\tau_N/w_N\to0$, imply uniformly over truncated configurations that
\[
 \frac{v_N-w_N}{w_N}=1+o(1),
 \qquad
 \frac{u_x-v_N}{w_N}=\frac{x}{N^{2/3}}(1+o(1)).
\]
Consequently, uniformly for $x,y\ge-L_N$ with $L_N=o(N^{2/3})$, one has
$u_x,u_y\ge3w_N/2$ for all sufficiently large $N$; increasing $x$ or $y$
only strengthens these lower bounds.

For fixed $B_w$, the $z$-closing segment can now be removed.  Indeed, for
$z=-B_z+{\rm i}\zeta$, $|\zeta|\le B_z+w_N$, one has
\begin{align*}
 \Re\bigl[-(z-u_x)^2\bigr]
 &=-(B_z+u_x)^2+\zeta^2\\
 &\le-(B_z+u_x)^2+(B_z+w_N)^2
 \le-cB_zw_N-cw_N^2.
\end{align*}
For fixed $N$ and $B_w$, the remaining rational factors have at most
polynomial growth in $B_z$.  Thus the integral over the closing segment is
bounded by
\[
 C_{N,B_w}(1+B_z)^{m_N}e^{-cB_zw_N/S_N}
\]
for some finite integer $m_N$, and tends to zero as $B_z\to\infty$, uniformly
for $w$ on the fixed truncated vertical contour.  We are left with the two
infinite $z$-rays.
For fixed $B_w$, Gaussian decay on the two $z$-rays is uniform for $w$
in the compact $w$-homotopy.  Hence the $z$-integral converges locally
uniformly in $w$ and defines an analytic function throughout the deformation
region.

\medskip

\noindent{{\it Deforming the $w$-contour}.} \
We next deform the upper and lower halves of the truncated vertical
$w$-contour to
\[
 w_N+d\ell_N+t\pm {\rm i}t,
 \qquad 0\le t\le B_w,
\]
using horizontal connectors at heights $\pm B_w$.  At height zero, the upper
and lower homotopies also contain the segment joining
$w_N+2d\ell_N$ to $w_N+d\ell_N$; the two copies have opposite orientations
and cancel when the upper and lower halves are recombined.  The deformed
$w$-contour remains in $\{\Re w\ge w_N+d\ell_N\}$, whereas the deformed
$z$-contour lies in $\{\Re z\le w_N-d\ell_N\}$.  Hence
$\Re(w-z)\ge2d\ell_N$ throughout this second homotopy.  No singularity is
crossed.

\medskip

\noindent{{\it Removing the $w$-connectors}.} \
We now show that the $w$-connectors disappear as $B_w\to\infty$.  Fix
$x,y\ge-L_N$.
For the upper $w$-connector write
$w=w_N+2d\ell_N+t+{\rm i}B_w$, with $0\le t\le B_w-d\ell_N$; the lower connector is
its conjugate.  Define $\beta_N(y)=w_N+2d\ell_N-u_y$.  Uniformly for $y\ge-L_N$,
$\beta_N(y)\le-cw_N$ for all sufficiently large $N$.  The
real part of the Gaussian exponent is
\[
 \Re (w-u_y)^2=(\beta_N(y)+t)^2-B_w^2.
\]
For $0\le t\le B_w/2$ this is at most $-cB_w^2$ once
$B_w\ge C(w_N+|\beta_N(y)|)$.  On
$B_w/2\le t\le B_w-d\ell_N$, convexity shows that its maximum is attained at
an endpoint.  At the right endpoint,
\[
 (\beta_N(y)+B_w-d\ell_N)^2-B_w^2
 =2B_w(\beta_N(y)-d\ell_N)+(\beta_N(y)-d\ell_N)^2,
\]
which is at most $-cB_ww_N$ once $B_w$ is sufficiently large (depending on the
fixed $y$).  Thus this connector is bounded by
$\exp\{-cB_w w_N/S_N\}$ times a polynomial in $B_w$.  Throughout the
homotopy, $|w-z|^{-1}\le(2d\ell_N)^{-1}$.  The local portions of the
infinite $z$-rays have finite length.  On a far ray write
$t=T_N^*+h$ and set $q=(-x)_+\le L_N$.  By
\eqref{eq:shifted-s-far},
\[
 -\Re\Phi_N(s)+x\Re s
 \le -cN-cN^{2/3}h+L_N(d+T_N^*+h)
 \le -c'N-c'N^{2/3}h,
\]
because $L_NT_N^*=o(N)$ and $L_N=o(N^{2/3})$.  Hence the absolute
$z$-integral is finite, uniformly over the allowed spatial range.  Dominated convergence
therefore removes both $w$-connectors as $B_w\to\infty$.  The tails of the
original vertical $w$-contour vanish by the same Gaussian estimate.  This is
the finite contour version of Johansson's deformation~\cite{Johansson2007}, with the separation
made explicit.

\medskip

\noindent{{\it Rescaling and absolute convergence}.} \
After these sequential limits, we make the changes of variables
$w=w_N+\ell_Nr$ and $z=w_N+\ell_Ns$.  The algebra preceding
\eqref{eq:Airy-double-contour} gives
\[
 \frac{\alpha_N}{\sqrt N}\frac{\ell_N^2}{S_N\ell_N}=1
\]
and the conjugation cancels every remaining constant in $x$ and $y$.
To see absolute convergence uniformly down to $x,y=-L_N$, write
$t=T_N^*+h$ on a far ray.  Since $L_N=o(N^{2/3})$ and
$T_N^*\asymp N^{1/3}$,
\[
 -cN-cN^{2/3}h+L_N(d+T_N^*+h)
 \le -c'N-c'N^{2/3}h
\]
for all sufficiently large $N$.  The local ray pieces are compact, and
nonnegative spatial variables only add decay.  Thus
\Cref{lem:shifted-contours} gives an integrable majorant.  We therefore obtain
\eqref{eq:scaled-double-contour}.  Finally,
\eqref{eq:contour-separation} permits
$(r-s)^{-1}=\int_0^\infty e^{-\rho(r-s)}\dd\rho$; the same decay estimates
and $\int_0^\infty e^{-2d\rho}\dd\rho<\infty$ justify Fubini's theorem and
the factorized integrals below.
\end{proof}

Once the kernel is on separated contours, the Cauchy factor can be
factorized.  This is the key step that converts a double-contour kernel
comparison into a comparison of two products of Hilbert--Schmidt
operators.


\subsection{Trace-norm comparison}


The strategy is now as follows.  The separated contours give
$(r-s)^{-1}=\int_0^\infty e^{-\rho(r-s)}\dd\rho$.  Interchanging
integrals writes each kernel as a product of two factors, one containing
the $s$-integral and one containing the $r$-integral.  We first prove
uniform Hilbert--Schmidt bounds for all four factors, then estimate the
differences of the finite-$N$ and Airy factors.

We repeatedly use the following elementary Laplace bound.

\begin{lemma}\label{lem:cubic-Laplace}
For every $c>0$ and integer $m\ge0$, there is a constant $C_m>0$
such that, for all $q\ge0$,
\begin{equation}\label{eq:cubic-Laplace}
 \int_0^\infty(1+t^m)e^{-ct^3+qt}\,\dd t
 \le C_m e^{C_mq^{3/2}}.
\end{equation}
\end{lemma}

\begin{proof}
Young's inequality gives $qt\le(c/2)t^3+Cq^{3/2}$.  After the cubic
term is absorbed, the remaining integral of
$(1+t^m)e^{-ct^3/2}$ is finite, which proves
\eqref{eq:cubic-Laplace}.
\end{proof}

We apply this bound after the contour factorization \eqref{eq:Cauchy-factor}.
It controls the positive exponential contribution created when the spatial
variable lies in $[-L,0]$.  In what follows, $\|\cdot\|_p$ denotes the
Schatten $p$-norm; see Simon~\cite[Chapter~2]{Simon2005} for the standard
definitions and ideal properties used below.

\begin{proposition}
\label{prop:trace-comparison}
There exist constants $c,C,c_*>0$ such that, whenever
\begin{equation}\label{eq:trace-window}
 L\ge1,
 \qquad
 L^3\le c_*\log N,
\end{equation}
one has
\begin{equation}\label{eq:trace-comparison}
 \sup_{\bxi\in\cG_N}
 \left\|
 \one_{(-L,\infty)}
 (\widetilde K_N-K_{\Ai})
 \one_{(-L,\infty)}
 \right\|_1
 \le CN^{-c}e^{CL^{3/2}}.
\end{equation}
Moreover,
\begin{equation}\label{eq:trace-norm-bounds}
 \left\|
 \one_{(-L,\infty)}K_{\Ai}\one_{(-L,\infty)}
 \right\|_1
 \le C(1+L^{3/2}),
\end{equation}
and, after decreasing $c_*$ if necessary,
\begin{equation}\label{eq:finite-kernel-trace-bound}
 \left\|
 \one_{(-L,\infty)}\widetilde K_N\one_{(-L,\infty)}
 \right\|_1
 \le C(1+L^{3/2}).
\end{equation}
\end{proposition}

\begin{proof}
Since $L^3\le c_*\log N$, one has $L=o(N^{2/3})$, so
\Cref{prop:scaled-contour-representation} applies on this window for all
sufficiently large $N$.  By \eqref{eq:contour-separation},
\begin{equation}\label{eq:Cauchy-factor}
 \frac1{r-s}=\int_0^\infty e^{-\rho(r-s)}\,\dd\rho.
\end{equation}
The piecewise estimates in \Cref{lem:shifted-contours} justify Fubini's
theorem.  Hence $\widetilde K_N=\mathcal A_N\mathcal B_N$, where
\begin{align}
 \mathcal A_N(x,\rho)
=\frac1{2\pi {\rm i}}\int_\Sigma
 e^{-\Phi_N(s)+xs+\rho s}\,\dd s,                      \label{eq:AN-factor}\\
 \mathcal B_N(\rho,y)
=\frac1{2\pi {\rm i}}\int_\Gamma
 e^{\Phi_N(r)-yr-\rho r}\,\dd r.                       \label{eq:BN-factor}
\end{align}
The Airy kernel has the identical factorization
$K_{\Ai}=\mathcal A_{\Ai}\mathcal B_{\Ai}$ with
$\Phi_N(\zeta)$ replaced by $\zeta^3/3$.  The operators $\mathcal A_N$
and $\mathcal A_{\Ai}$ map $L^2(0,\infty)$ to $L^2(-L,\infty)$, whereas
$\mathcal B_N$ and $\mathcal B_{\Ai}$ map in the opposite direction.
Their products therefore act on $L^2(-L,\infty)$.  The estimates below
show that each factor is Hilbert--Schmidt.  Consequently,
the ideal property of Schatten classes in
Simon~\cite[Chapter~2]{Simon2005} shows that
both products are trace class, which also justifies the operator Fredholm
determinants used in the sequel.

We first estimate the factors.  Parametrize a branch of $\Sigma$ by
$s=-d-t\pm {\rm i}t$.  If $x\ge0$, then $x\Re s\le-dx$.  If
$x=-q\in[-L,0)$, then $x\Re s=q(d+t)$.  The local and far bounds of
\Cref{lem:shifted-contours} imply
\begin{align}
 |\mathcal A_N(x,\rho)|
 &\le Ce^{-d\rho}
 \begin{cases}
 e^{-dx},&x\ge0,\\
 e^{dq}\displaystyle\int_0^\infty e^{-ct^3+qt}\,\dd t,
 &x=-q\in[-L,0),
 \end{cases}                                           \label{eq:A-prebound}
\end{align}
where the far finite-$N$ portion is exponentially smaller and has been
absorbed into the constant.  \Cref{lem:cubic-Laplace} gives
\begin{equation}\label{eq:A-pointwise}
 |\mathcal A_N(x,\rho)|
 \le Ce^{CL^{3/2}}e^{-d\rho}
 \left(
 \one_{[-L,0]}(x)+e^{-dx}\one_{[0,\infty)}(x)
 \right).
\end{equation}
The same calculation on $\Gamma$ gives
\begin{equation}\label{eq:B-pointwise}
 |\mathcal B_N(\rho,y)|
 \le Ce^{CL^{3/2}}e^{-d\rho}
 \left(
 \one_{[-L,0]}(y)+e^{-dy}\one_{[0,\infty)}(y)
 \right).
\end{equation}
The Airy factors satisfy the same estimates.  For example,
\begin{align*}
 \|\mathcal A_N\|_2^2
 \le C e^{CL^{3/2}}
 \int_0^\infty e^{-2d\rho}\dd\rho
 \left(
 \int_{-L}^0\dd x+
 \int_0^\infty e^{-2dx}\dd x
 \right)\le C(1+L)e^{CL^{3/2}}.
\end{align*}
The same calculation applies to the other three factors.  Taking square
roots and enlarging $C$ gives
\begin{equation}\label{eq:HS-uniform}
 \|\mathcal A_N\|_2+\|\mathcal B_N\|_2
 +\|\mathcal A_{\Ai}\|_2+\|\mathcal B_{\Ai}\|_2
 \le Ce^{CL^{3/2}},
\end{equation}
since $(1+L)^{1/2}\le Ce^{CL^{3/2}}$ for $L\ge1$.

We next compare the factors.  Split every contour at $t=N^\eta$.  On the
local part of $\Sigma$, for instance,
\eqref{eq:shifted-s-comparison} gives, with $x=-q\in[-L,0]$,
\[
 |(\mathcal A_N-\mathcal A_{\Ai})(x,\rho)|
 \le CN^{-c}e^{-d\rho+dq}
 \int_0^{N^\eta}(1+t^4)e^{-ct^3+qt}\dd t.
\]
For $x\ge0$ the same expression has the additional factor $e^{-dx}$ and no
positive $qt$ term.  The analogous estimate holds for
$\mathcal B_N-\mathcal B_{\Ai}$.  The case $m=4$ of
\Cref{lem:cubic-Laplace}, followed by the explicit Hilbert--Schmidt
integration above, therefore yields
\begin{equation}\label{eq:factor-difference-local}
 \|\mathcal A_N-\mathcal A_{\Ai}\|_2
 +\|\mathcal B_N-\mathcal B_{\Ai}\|_2
 \le CN^{-c}e^{CL^{3/2}}.
\end{equation}
On the complementary contour pieces, the last assertion of
\Cref{lem:shifted-contours} gives a pointwise bound
$Ce^{-cN^{3\eta}}e^{-d\rho}$, with the same positive half-line spatial
decay as above.  Its Hilbert--Schmidt norm is at most
$Ce^{-cN^{3\eta}/2}$ after absorbing the factor $(1+L)^{1/2}$.  This is
smaller than the right-hand side of
\eqref{eq:factor-difference-local}.  Therefore
\begin{equation}\label{eq:factor-difference}
 \|\mathcal A_N-\mathcal A_{\Ai}\|_2
 +\|\mathcal B_N-\mathcal B_{\Ai}\|_2
 \le CN^{-c}e^{CL^{3/2}}.
\end{equation}

The ideal property of Schatten norms in
Simon~\cite[Chapter~2]{Simon2005} gives
\begin{align*}
 \|\widetilde K_N-K_{\Ai}\|_1
 &\le
 \|\mathcal A_N-\mathcal A_{\Ai}\|_2\|\mathcal B_N\|_2
 +\|\mathcal A_{\Ai}\|_2
  \|\mathcal B_N-\mathcal B_{\Ai}\|_2.
\end{align*}
Equations \eqref{eq:HS-uniform} and \eqref{eq:factor-difference} prove
\eqref{eq:trace-comparison} after renaming constants.

Finally, the Airy kernel is positive because, for every
$f\in L^2(-L,\infty)$,
\[
 \left\langle f,
 \one_{(-L,\infty)}K_{\Ai}\one_{(-L,\infty)}f\right\rangle
 =\int_0^\infty
 \Big|\int_{-L}^\infty f(x)\Ai(x+\lambda)\,\dd x\Big|^2\dd\lambda
 \ge0.
\]
It is trace class by the factorization above, so its trace norm equals its
trace.  Since its kernel is continuous, the diagonal trace formula in
Simon~\cite[Theorem~2.12]{Simon2005} then gives
\[
 \left\|
 \one_{(-L,\infty)}K_{\Ai}\one_{(-L,\infty)}
 \right\|_1
 =\int_{-L}^\infty K_{\Ai}(x,x)\,\dd x.
\]
To verify the required diagonal bounds, use the identity
\[
 K_{\Ai}(x,x)=\Ai'(x)^2-x\Ai(x)^2,
\]
which follows from integration by parts and the Airy equation
$\Ai''(x)=x\Ai(x)$.  The classical pointwise estimates in Olver et
al.~\cite[Section~9.7]{OlverEtAl2010}
\[
 |\Ai(-u)|\le C(1+u)^{-1/4},\qquad
 |\Ai'(-u)|\le C(1+u)^{1/4},\qquad u\ge0,
\]
and
\[
 |\Ai(u)|\le C(1+u)^{-1/4}e^{-2u^{3/2}/3},
 \qquad
 |\Ai'(u)|\le C(1+u)^{1/4}e^{-2u^{3/2}/3},
 \qquad u\ge0,
\]
therefore give
\[
 K_{\Ai}(x,x)\le C(1+|x|^{1/2}),\quad x\le0,
 \qquad
 K_{\Ai}(x,x)\le Ce^{-cx^{3/2}},\quad x\ge0.
\]
Integrating these estimates proves \eqref{eq:trace-norm-bounds}.  Under
\eqref{eq:trace-window}, after choosing $c_*$ sufficiently small,
$N^{-c}e^{CL^{3/2}}=o(1)$.  The finite-$N$ trace-norm bound follows from
\eqref{eq:trace-comparison} and the triangle inequality.
\end{proof}

The trace-norm estimate is converted next into a relative comparison of the
conditional distribution with the Tracy--Widom law.


\section{Conditional Tracy--Widom tails}\label{sec:conditional-tails}


This section converts the operator estimate of \Cref{sec:airy} into
estimates for conditional probabilities.  There are two comparisons.  First,
trace-norm continuity gives an absolute comparison of Fredholm determinants.
The second, a relative one, compares the error with the Tracy--Widom tail itself.
It is here that the logarithmic restriction on the Airy window enters.

For a deterministic truncated configuration
$\mathbf x=(x_1,\ldots,x_N)\in[-\tau_N,\tau_N]^N$, let
\begin{equation}\label{eq:HN}
 \mathcal H_N(t;\mathbf x)
 =\Prob_{\mathbf x}(\chi_N\le t)
\end{equation}
denote the conditional distribution function, where $\Prob_{\mathbf x}$ is
the law of the GUE perturbation with the diagonal fixed at $\mathbf x$, and
under this law
\[
 \chi_N=\chi_N(\mathbf x)
 :=\frac{\sqrt N}{\alpha_N}\bigl(\lambda_N-v_N(\mathbf x)\bigr).
\]
For all sufficiently large $N$, $w_N>2\tau_N$, so the denominators defining
$v_N(\mathbf x)$ are uniformly nonsingular on the truncated cube.  This
definition also chooses a regular conditional distribution for every
deterministic $\mathbf x$; evaluated at the random configuration, it
satisfies
$\mathcal H_N(t;\bxi)=\Prob(\chi_N\le t\mid\cF_N)$ almost surely.

Starting from \eqref{eq:conditional-Fredholm}, set the spectral threshold to 
$v_N(\mathbf x)+\alpha_Nt/\sqrt N$.  The change of variables and the
termwise conjugation from \Cref{sec:airy} then give
\begin{equation}\label{eq:HN-Fredholm}
 \mathcal H_N(t;\mathbf x)
 =\det\!\left(
 I-\widetilde K_N(\mathord{\cdot},\mathord{\cdot};\mathbf x)
 \right)_{L^2(t,\infty)}.
\end{equation}
For an arbitrary deterministic truncated configuration, the determinant in
\eqref{eq:HN-Fredholm} is interpreted through the finite Fredholm series in
\eqref{eq:Fredholm-series}, whose terms are unchanged by the rescaling and
conjugation.  On $\cG_N$ and in the windows considered below, the trace-class
results of \Cref{sec:airy} show that it is also the usual operator Fredholm
determinant.

The trace-norm comparison now gives a relative Fredholm determinant estimate,
and therefore the required conditional Tracy--Widom tails.

\begin{proposition}
\label{prop:conditional-TW}
There exist constants $c,C,c_*>0$ such that, whenever
$L\ge1$ and $L^3\le c_*\log N$,
\begin{equation}\label{eq:determinant-comparison}
 \sup_{\mathbf x\in\cG_N}\sup_{{t\ge-L}}
 |\mathcal H_N(t;\mathbf x)-F_2(t)|
 \le CN^{-c}e^{CL^{3/2}}.
\end{equation}
Moreover, for every compact set $K\subset(0,\infty)$, the following two
estimates hold uniformly for $\mathbf x\in\cG_N$ and $u\in K$:
\begin{align}
 \sup_{\mathbf x\in\cG_N}\sup_{u\in K}
 \left|
 \log\Prob_{\mathbf x}(\chi_N>a_Nu)
 +\frac43(a_Nu)^{3/2}
 \right|
 =o(a_N^{3/2}),                                      \label{eq:conditional-right-tail}\\
 \sup_{\mathbf x\in\cG_N}\sup_{u\in K}
 \left|
 \log\Prob_{\mathbf x}(\chi_N\le-a_Nu)
 +\frac1{12}(a_Nu)^3
 \right|
 =o(a_N^3).                                          \label{eq:conditional-left-tail}
\end{align}
In particular, evaluating a fixed version of the conditional probabilities at
$\mathbf x=\bxi$ gives the same deterministic bounds on $\cG_N$.  By
\Cref{lem:good-event}, the complement $\cG_N^c$ is superexponentially
negligible at every speed used in \Cref{sec:phase-proof}.
\end{proposition}

\begin{proof}
\noindent{{\it Determinant comparison}.} \
For trace-class operators $A$ and $B$, the determinant continuity estimate in
Simon~\cite[Theorem~3.4]{Simon2005} gives
\begin{equation}\label{eq:det-continuity}
 |\det(I-A)-\det(I-B)|
 \le\|A-B\|_1
 \exp\{1+\|A\|_1+\|B\|_1\}.
\end{equation}
For $t\ge-L$, let $A_{N,t}$ and $A_{\Ai,t}$ denote the compressions of
$\widetilde K_N$ and $K_{\Ai}$ to $L^2(t,\infty)$.  They are also
compressions of the corresponding operators on $L^2(-L,\infty)$, and
multiplication by a projection does not increase trace norm.  Hence
\Cref{prop:trace-comparison} gives, uniformly for $t\ge-L$ and
$\mathbf x\in\cG_N$,
\[
 \|A_{N,t}-A_{\Ai,t}\|_1\le CN^{-c}e^{CL^{3/2}},
 \qquad
 \|A_{N,t}\|_1+\|A_{\Ai,t}\|_1\le C(1+L^{3/2}).
\]
Apply \eqref{eq:det-continuity} with $A=A_{N,t}$ and $B=A_{\Ai,t}$.  Since
$e^{C(1+L^{3/2})}$ has the same form as the factor already present in the
trace-norm error, renaming $c,C$ proves
\eqref{eq:determinant-comparison}.

\medskip

\noindent{{\it The relative error test}.} \ 
Fix a compact interval $[u_0,u_1]\subset(0,\infty)$ containing $K$.
For the right tail, apply \eqref{eq:determinant-comparison} with $L=1$;
then $t=a_Nu\ge-1$, and the resulting determinant error satisfies
$\varepsilon_N^+\le CN^{-c}$.  We first compare this error with the right
Tracy--Widom  tail, uniformly for $u\in[u_0,u_1]$. 

For $t=a_Nu$, the Tracy--Widom right-tail expansion
\eqref{eq:TW-right-intro} gives
\[
 1-F_2(t)
 =\exp\left\{-\frac43t^{3/2}+O(\log t)\right\}.
\]
Because $u$ is bounded away from zero and infinity,
\begin{align*}
 \log\frac{{\varepsilon_N^+}}{1-F_2(a_Nu)}
 &\le {-c\log N}
 +\frac43(a_Nu)^{3/2}+O(\log a_N)\\
 &=-c\log N+o(\log N),
\end{align*}
uniformly for $u\in[u_0,u_1]$.  Here we used
$a_N^{3/2}=o(\log N)$, which follows from \Cref{ass:scale}.  Thus
$\varepsilon_N^+=o(1-F_2(a_Nu))$ uniformly on the stated interval, and the
determinant comparison yields
\[
 1-\mathcal H_N(a_Nu;\mathbf x)
 =(1-F_2(a_Nu))(1+o(1))
\]
uniformly on $\cG_N$.  Taking logarithms and using
$O(\log a_N)=o(a_N^{3/2})$ proves
\eqref{eq:conditional-right-tail}.

For the left tail, take $L=2u_1a_N$.  Since $a_N^3=o(\log N)$,
$L^3\le c_*\log N$ for all sufficiently large $N$.  After renaming
constants, a deterministic error majorant is
\begin{equation}\label{eq:det-error-log}
\varepsilon_N^-:=CN^{-c}e^{Ca_N^{3/2}}
 =\exp\{-c\log N+o(\log N)\}.
\end{equation}
For $t=-a_Nu$, the left-tail expansion
\eqref{eq:TW-left-intro} gives
\[
 F_2(-a_Nu)
 =\exp\left\{-\frac1{12}(a_Nu)^3+O(\log a_N)\right\}.
\]
The relative error check is now
\begin{align*}
 \log\frac{\varepsilon_N^-}{F_2(-a_Nu)}
 &\le-c\log N+C a_N^{3/2}
 +\frac1{12}(a_Nu)^3+O(\log a_N)\\
 &=-c\log N+o(\log N),
\end{align*}
uniformly for $u\in[u_0,u_1]$.  Hence
$\varepsilon_N^-=o(F_2(-a_Nu))$,  and
\[
 \mathcal H_N(-a_Nu;\mathbf x)
 =F_2(-a_Nu)(1+o(1))
\]
uniformly on $\cG_N$.  Taking logarithms and using
$O(\log a_N)=o(a_N^3)$ proves
\eqref{eq:conditional-left-tail}.
\end{proof}

This comparison also makes the logarithmic-window restriction quantitative.

\begin{remark}\label{rem:log-window}
The left tail imposes the binding restriction.  For the right tail, the
relative error test compares $\varepsilon_N^+$ with
$\exp\{-c a_N^{3/2}\}$, so it would be enough to assume
$a_N^{3/2}=o(\log N)$.  For the left tail, the target probability is instead
\[
F_2(-a_Nx)=\exp\{-a_N^3x^3/12+O(\log a_N)\},
\qquad x>0,
\]
so the ratio test also contains a factor $\exp\{c a_N^3\}$.  This is why the
binding condition is $a_N^3=o(\log N)$.  A larger window would require an
exponentially accurate finite-$N$ determinant expansion, rather than the
polynomially accurate trace-norm estimate obtained here.  This is the
computational content of \Cref{rem:log-window-statement}.
\end{remark}


\section{Proof of the phase diagram}\label{sec:phase-proof}


The final step combines the Airy and Gaussian contributions.  The
displacement $s_N$ is $\cF_N$-measurable, whereas the conditional law of the
Airy fluctuation $\chi_N$ is described by the uniform estimates of
\Cref{prop:conditional-TW}.  Both are tied to the same diagonal
configuration, so we condition on $\cF_N$.  On the event $\cG_N$, the conditional
Airy bounds are deterministic and uniform, while the displacement is governed
by its Gaussian MDP.  The following conditional convolution principle
packages this step; in the application, $X_N=\chi_N$ and $Y_N=s_N$.

\begin{lemma}
\label{lem:conditional-convolution}
Let $X_N+Y_N$ be a decomposition with $Y_N$ measurable with respect to a
$\sigma$-algebra $\cF_N$.  Let $a_N\to\infty$, and let $\alpha_N a_N\to\infty$.
Assume the following.

\begin{enumerate}[label=\textup{(\roman*)},leftmargin=2.7em]
\item There are events $\cG_N\in\cF_N$ whose complements are
superexponentially small at the speeds
$a_N^{3/2}$, $a_N^3$, and $\alpha_N^2a_N^2$.
\item Uniformly on $\cG_N$ and uniformly for $u$ in compact subsets of
$(0,\infty)$,
\begin{align*}
 \log\Prob(X_N>a_Nu\mid\cF_N)
 =-a_N^{3/2}\left(\frac43u^{3/2}+o(1)\right),\\
 \log\Prob(X_N\le-a_Nu\mid\cF_N)
 =-a_N^3\left(\frac1{12}u^3+o(1)\right).
\end{align*}
\item The variables $Y_N/a_N$ satisfy an MDP with speed
$\alpha_N^2a_N^2$ and rate $v^2/(2\sigma^2)$.
\end{enumerate}

Here and below, fixed versions of the conditional probabilities are used.
Uniformity on $\cG_N$ means that the displayed deterministic error bounds hold
for almost every outcome in $\cG_N$.

Then the six conclusions of \Cref{thm:main} hold with $Z_N$ replaced by
$X_N+Y_N$.
\end{lemma}

\begin{proof}
We write
\[
 \mathsf{t}_N^+=a_N^{3/2},
 \qquad
 \mathsf{t}_N^-=a_N^3,
 \qquad
 \mathsf{g}_N=\alpha_N^2a_N^2.
\]
Thus $\mathsf{g}_N/\mathsf{t}_N^+=\theta_N^+$ and
$\mathsf{g}_N/\mathsf{t}_N^-=\theta_N^-$.  We write
$y_+=\max\{y,0\}$ for the positive part of $y$.  Fix $x>0$.  All auxiliary
parameters below are chosen independently of $N$.

The six cases follow a common pattern.  For an upper bound, we localize or
truncate the displacement $Y_N/a_N$ and combine the conditional Airy and
Gaussian rate contributions.  For a lower bound, we prescribe one
near-optimal displacement and then require the conditional Airy variable to
complete the deviation.  The event $\cG_N$ allows the conditional estimates
to be used as deterministic bounds throughout.  We first state three
consequences of the assumptions that will be used repeatedly.

\medskip
\noindent{{\it Three uniform estimates.}} \
Applying the MDP closed-set upper bound to
$\{v:|v|\ge M\}$ gives
\[
 \limsup_{N\to\infty}\frac1{\mathsf{g}_N}
 \log\Prob(|Y_N|>Ma_N)\le-\frac{M^2}{2\sigma^2}.
\]
Letting $M\to\infty$ gives exponential tightness:
\begin{equation}\label{eq:Y-exp-tight}
 \lim_{M\to\infty}\limsup_{N\to\infty}
 \frac1{\mathsf{g}_N}\log\Prob(|Y_N|>Ma_N)=-\infty.
\end{equation}
If $I$ is a fixed nonempty open interval, then
\begin{equation}\label{eq:Y-good-lower}
 \liminf_{N\to\infty}\frac1{\mathsf{g}_N}
 \log\Prob(Y_N/a_N\in I,\ \cG_N)
 \ge-\inf_{v\in I}\frac{v^2}{2\sigma^2}.
\end{equation}
Indeed, the MDP gives the same lower bound without $\cG_N$, and
the subtraction can be quantified as follows.  Define 
$c_I=\inf_{v\in I}v^2/(2\sigma^2)$.  For every $\varepsilon>0$ and all
sufficiently large $N$, the MDP lower bound and superexponential estimate give
\[
 \Prob(Y_N/a_N\in I)\ge e^{-\mathsf{g}_N(c_I+\varepsilon)},
 \qquad
 \Prob(\cG_N^c)\le e^{-\mathsf{g}_N(c_I+2\varepsilon)}.
\]
Consequently,
\[
 \Prob(Y_N/a_N\in I,\cG_N)
 \ge e^{-\mathsf{g}_N(c_I+\varepsilon)}
 \bigl(1-e^{-\varepsilon\mathsf{g}_N}\bigr),
\]
which proves \eqref{eq:Y-good-lower} because $\mathsf{g}_N=(\alpha_Na_N)^2
\to\infty$ and $\varepsilon$ is arbitrary.  The corresponding closed-set
upper bound for an intersection with $\cG_N$ follows directly from the MDP.

Finally, the uniform conditional estimates imply that, for every compact
$K\subset(0,\infty)$, there is a deterministic $\eps_N(K)\downarrow0$ such
that on $\cG_N$ the following bounds hold for $u\in K$; we suppress the
dependence of $\eps_N$ on $K$:
\begin{align}
 e^{-\mathsf{t}_N^+(\frac43u^{3/2}+\eps_N)}
 &\le \Prob(X_N>a_Nu\mid\cF_N)
 \le e^{-\mathsf{t}_N^+(\frac43u^{3/2}-\eps_N)},                 \label{eq:uniform-cond-right}\\
 e^{-\mathsf{t}_N^-(\frac1{12}u^3+\eps_N)}
 &\le \Prob(X_N\le-a_Nu\mid\cF_N)
 \le e^{-\mathsf{t}_N^-(\frac1{12}u^3-\eps_N)}.                \label{eq:uniform-cond-left}
\end{align}
These two-sided bounds are the exact form in which
\Cref{prop:conditional-TW} enters.  They also imply that a moderate
Airy fluctuation is conditionally typical: for every fixed $\eta>0$,
\begin{align}
 \Prob(X_N>-\eta a_N\mid\cF_N)
 &=1-o(1),                                             \label{eq:conditional-typical-right}\\
 \Prob(X_N\le \eta a_N\mid\cF_N)
 &=1-o(1),                                             \label{eq:conditional-typical-left}
\end{align}
uniformly on $\cG_N$.  Indeed, their complementary probabilities are
controlled by
\eqref{eq:uniform-cond-left} and \eqref{eq:uniform-cond-right}, respectively.
More precisely, they are bounded by
$\exp\{-c_\eta \mathsf{t}_N^-+o(\mathsf{t}_N^-)\}$ and
$\exp\{-c_\eta \mathsf{t}_N^++o(\mathsf{t}_N^+)\}$, respectively.  These
observations provide the required uniform conditional lower bound after
restriction to a rare $Y_N$ event.

\medskip
\noindent\emph{1. Right-tail crossover: $\theta_N^+\to\theta\in(0,\infty)$.}
We work at speed $\mathsf{t}_N^+$.  In this regime the conditional Airy
and Gaussian rate contributions are comparable, so both enter the
variational problem.  For the upper bound we
first localize the scaled displacement $Y_N/a_N$.  Fix $M>x+1$ and
cover $[-M,M]$ by finitely many closed intervals
$I_k=[v_k,v_k+\eta]$, with the endpoint intervals truncated at $\pm M$.  On
$\{Y_N/a_N\in I_k\}\cap\cG_N$,
\[
 \Prob(X_N+Y_N>a_Nx\mid\cF_N)
 \le
 \Prob\bigl(X_N>a_N(x-v_k-\eta)\mid\cF_N\bigr).
\]
If $x-v_k-\eta\le0$, the last probability is bounded by one and carries zero
$X_N$ contribution.  For the finitely many remaining intervals, the finite
set of positive thresholds has a positive minimum, so the compact uniform
conditional right-tail estimate applies.
For fixed $M$ and mesh size $\eta$, the uniform conditional estimate, the
MDP upper bound for $Y_N/a_N$, and a finite union bound give the following
estimate.  Finite sums do not change the exponential rate.  Indeed, for each
cell,
\begin{align*}
 &\Prob(X_N+Y_N>a_Nx,\ Y_N/a_N\in I_k,\ \cG_N)\\
 &\quad\le \Prob(Y_N/a_N\in I_k)
 \exp\left\{-\mathsf{t}_N^+
 \left(\frac43(x-v_k-\eta)_+^{3/2}-o(1)\right)\right\},
\end{align*}
where the error is uniform over the finitely many cells.  The MDP upper bound
is applied to $\overline I_k$, and
$\mathsf{g}_N/\mathsf{t}_N^+\to\theta$.  Hence
\begin{align*}
 &\limsup_{N\to\infty}\frac1{\mathsf{t}_N^+}
 \log\Prob(X_N+Y_N>a_Nx,\ |Y_N|\le Ma_N,\ \cG_N)\\
 &\qquad\le -\min_k\left\{
 \frac43(x-v_k-\eta)_+^{3/2}
 +\theta\inf_{v\in\overline I_k}\frac{v^2}{2\sigma^2}
 \right\}.
\end{align*}
The complement $\cG_N^c$ is superexponentially negligible.  By
\eqref{eq:Y-exp-tight} and $\mathsf{g}_N/\mathsf{t}_N^+\to\theta>0$,
\[
 \lim_{M\to\infty}\limsup_{N\to\infty}
 \frac1{\mathsf{t}_N^+}\log\Prob(|Y_N|>Ma_N)=-\infty.
\]
For fixed $M$, let the mesh size tend to zero.  Continuity of the objective
gives the corresponding infimum over $[-M,M]$.  We may then let $M\to\infty$:
the preceding exponential tightness estimate removes the complement, and
coercivity of the quadratic term makes the compact infima converge to the
infimum over $\R$.  Hence
\begin{align}
 \limsup_{N\to\infty}\frac1{\mathsf{t}_N^+}
 \log\Prob(X_N+Y_N>a_Nx)
 \le-
 \inf_{v\in\R}
 \left\{
 \frac43(x-v)_+^{3/2}
 +\frac{\theta v^2}{2\sigma^2}
 \right\}.                                             \label{eq:right-cross-upper}
\end{align}
The minimizer in \eqref{eq:right-cross-upper} lies in $[0,x]$: a negative
$v$ increases both terms relative to $v=0$, while a value $v>x$ can be
replaced by $x$ and strictly lowers the quadratic term.  Setting $u=x-v$
identifies the right-hand side with
$-J_+^{(\theta)}(x)$.

For the lower bound, we reverse the strategy: choose a division $u+v=x$ of
the required deviation, force the displacement to lie near $v$, and require
the conditional Airy variable to exceed a slightly larger threshold.  Take
$u\in(0,x]$, set $v=x-u$, and fix $\eta>0$.  On the event
$Y_N/a_N\in(v-\eta,v+\eta)$, it is enough to require
$X_N>a_N(u+2\eta)$.  Therefore
\begin{align*}
 \Prob(X_N+Y_N>a_Nx)
 &\ge
 \E\Bigl[
 \one_{\{Y_N/a_N\in(v-\eta,v+\eta)\}\cap\cG_N}
 \Prob(X_N>a_N(u+2\eta)\mid\cF_N)
 \Bigr].
\end{align*}
By \eqref{eq:uniform-cond-right}, the conditional factor is at least
\[
 \exp\left\{-\mathsf{t}_N^+
 \left(\frac43(u+2\eta)^{3/2}+o(1)\right)\right\}
\]
uniformly on $\cG_N$.  Equation \eqref{eq:Y-good-lower}, together with
$\mathsf{g}_N/\mathsf{t}_N^+\to\theta$, therefore gives
\[
 \liminf_{N\to\infty}\frac1{\mathsf{t}_N^+}
 \log\Prob(X_N+Y_N>a_Nx)
 \ge
 -\frac43(u+2\eta)^{3/2}
 -\frac{\theta v^2}{2\sigma^2}.
\]
For the endpoint $u=0$, choose instead
$Y_N/a_N\in(x+\eta,x+2\eta)$ and require
$X_N>-\eta a_N$.  By
\eqref{eq:conditional-typical-right}, this requirement has conditional
probability $1-o(1)$ uniformly on $\cG_N$, so the endpoint has zero
$X_N$ contribution.  More explicitly, \eqref{eq:Y-good-lower} yields
\[
 \liminf_{N\to\infty}\frac1{\mathsf{t}_N^+}
 \log\Prob(X_N+Y_N>a_Nx)
 \ge-\frac{\theta(x+\eta)^2}{2\sigma^2}.
\]
Sending $\eta\downarrow0$ and optimizing over $u\in[0,x]$ proves the right
crossover lower bound.

\medskip
\noindent\emph{2. Right Tracy--Widom regime: $\theta_N^+\to\infty$.}
Here a displacement of order $a_N$ is exponentially more expensive than the
conditional Airy fluctuation.  The proof therefore localizes $Y_N/a_N$ near
zero.  For $0<\eta<x$, the union bound gives
\begin{align}
 \Prob(X_N+Y_N>a_Nx)
 &\le
 \Prob(Y_N>\eta a_N)
 +\E\left[
 \one_{\cG_N}
 \Prob(X_N>a_N(x-\eta)\mid\cF_N)
 \right]
 +\Prob(\cG_N^c).                                      \label{eq:right-TW-upper}
\end{align}
The MDP upper bound for the closed set
$\{v:|v|\ge\eta\}$ gives
\[
 \limsup_{N\to\infty}\frac1{\mathsf{t}_N^+}
 \log\Prob(|Y_N|>\eta a_N)=-\infty,
\]
because $\mathsf{g}_N/\mathsf{t}_N^+=\theta_N^+\to\infty$.  The good-event
complement is superexponentially small at the same speed.  Thus the first and
third terms in \eqref{eq:right-TW-upper} are negligible, while the conditional
right-tail estimate gives the upper exponent
$-4(x-\eta)^{3/2}/3$.

For the lower bound,
\begin{align}
 \Prob(X_N+Y_N>a_Nx)
 &\ge
 \E\left[
 \one_{\{|Y_N|\le\eta a_N\}\cap\cG_N}
 \Prob(X_N>a_N(x+\eta)\mid\cF_N)
 \right].                                             \label{eq:right-TW-lower}
\end{align}
The preceding superexponential bound, together with the corresponding bound
for $\cG_N^c$, gives
\[
 \frac1{\mathsf{t}_N^+}\log
 \Prob(|Y_N|\le\eta a_N,\ \cG_N)\longrightarrow0.
\]
Multiplying this event probability by the uniform conditional lower bound in
\eqref{eq:right-TW-lower} yields the exponent
$-4(x+\eta)^{3/2}/3$.  Sending $\eta\downarrow0$ proves (R1).

\medskip
\noindent\emph{3. Right Gaussian regime: $\theta_N^+\to0$.}
We work at speed $\mathsf{g}_N$.  At this speed, a moderate Airy
fluctuation is superexponentially unlikely, and the displacement produces the
right deviation.  For $0<\eta<x$,
\begin{align}
 \Prob(X_N+Y_N>a_Nx)
 &\le
 \Prob(Y_N>a_N(x-\eta))
 +\E\left[
 \one_{\cG_N}\Prob(X_N>\eta a_N\mid\cF_N)
 \right]
 +\Prob(\cG_N^c).                                      \label{eq:right-G-upper}
\end{align}
Uniformly on $\cG_N$, the second term is at most
\[
 \exp\left\{-\mathsf{t}_N^+
 \left(\frac43\eta^{3/2}+o(1)\right)\right\}.
\]
Since $\mathsf{t}_N^+/\mathsf{g}_N=1/\theta_N^+\to\infty$, this term and
$\Prob(\cG_N^c)$ are superexponentially small at speed $\mathsf{g}_N$.  Thus
only the Gaussian term contributes, and the MDP upper bound gives at most
$-(x-\eta)^2/(2\sigma^2)$.

For the lower bound, use the event
$Y_N/a_N\in(x+\eta,x+2\eta)$ and require
$X_N>-\eta a_N$:
\begin{align}
 \Prob(X_N+Y_N>a_Nx)
 &\ge
 \E\left[
 \one_{\{Y_N/a_N\in(x+\eta,x+2\eta)\}\cap\cG_N}
 \Prob(X_N>-\eta a_N\mid\cF_N)
 \right].                                             \label{eq:right-G-lower}
\end{align}
By \eqref{eq:conditional-typical-right}, the conditional factor in
\eqref{eq:right-G-lower} equals $1-o(1)$ uniformly on $\cG_N$.  Applying
\eqref{eq:Y-good-lower} to the interval $(x+\eta,x+2\eta)$ therefore gives
the lower exponent $-(x+\eta)^2/(2\sigma^2)$.  Letting
$\eta\downarrow0$ proves (R3).

\medskip
\noindent\emph{4. Left-tail crossover: $\theta_N^-\to\theta\in(0,\infty)$.}
We now work at speed $\mathsf{t}_N^-$.  The left-tail argument has the same
structure as the right-tail crossover, but a negative displacement now
assists the event.  Use the same finite closed cover for $Y_N/a_N$.  If
$Y_N/a_N\in[v_k,v_k+\eta]$, then
\[
 \Prob(X_N+Y_N\le-a_Nx\mid\cF_N)
 \le
 \Prob\bigl(X_N\le-a_N(x+v_k)\mid\cF_N\bigr).
\]
When $x+v_k\le0$, this is bounded by one.  For the finitely many remaining
intervals, the finite set of positive thresholds has a positive minimum, so
the compact uniform conditional left-tail estimate applies.
More precisely,
\begin{align*}
 &\Prob(X_N+Y_N\le-a_Nx,\ Y_N/a_N\in I_k,\ \cG_N)\\
 &\quad\le \Prob(Y_N/a_N\in I_k)
 \exp\left\{-\mathsf{t}_N^-
 \left(\frac1{12}(x+v_k)_+^3-o(1)\right)\right\}.
\end{align*}
The error is uniform over the finite cover.  The MDP upper bound on
$\overline I_k$ and $\mathsf{g}_N/\mathsf{t}_N^-\to\theta$ therefore give, for
fixed $M$ and $\eta$,
\begin{align*}
 \limsup_{N\to\infty}\frac1{\mathsf{t}_N^-}
 \log\Prob(X_N+Y_N\le-a_Nx,\ |Y_N|\le Ma_N,\ \cG_N)
 \le -\min_k\left\{
 \frac1{12}(x+v_k)_+^3
 +\theta\inf_{v\in\overline I_k}\frac{v^2}{2\sigma^2}
 \right\}.
\end{align*}
The good-event complement is negligible.  For fixed $M$, letting the mesh
size $\eta\downarrow0$ gives the infimum over $[-M,M]$.  Since
$\mathsf{g}_N/\mathsf{t}_N^-\to\theta>0$, exponential tightness in
\eqref{eq:Y-exp-tight} removes $|Y_N|>Ma_N$ as $M\to\infty$, while coercivity
of the quadratic term makes the compact infima converge to the infimum over
$\R$.  Therefore
\begin{align}
 \limsup_{N\to\infty}\frac1{\mathsf{t}_N^-}
 \log\Prob(X_N+Y_N\le-a_Nx)
 \le-
 \inf_{v\in\R}
 \left\{
 \frac1{12}(x+v)_+^3
 +\frac{\theta v^2}{2\sigma^2}
 \right\}.                                             \label{eq:left-cross-upper}
\end{align}
The minimizer lies in $[-x,0]$: values $v>0$ increase both terms relative to
$v=0$, and values $v<-x$ can be replaced by $-x$ without increasing the
Tracy--Widom term while lowering the quadratic term.  Setting $u=x+v$ reduces
\eqref{eq:left-cross-upper} to $-J_-^{(\theta)}(x)$.

For the lower bound, we again prescribe a division of the deviation.  Choose
$u\in(0,x]$, set $v=u-x\in[-x,0]$, and restrict to
$Y_N/a_N\in(v-\eta,v+\eta)$.  If in addition
$X_N\le-a_N(u+2\eta)$, then $X_N+Y_N\le-a_Nx$.  The conditional lower
estimate \eqref{eq:uniform-cond-left} and
\eqref{eq:Y-good-lower}, with $\mathsf{g}_N/\mathsf{t}_N^-\to\theta$, give
\[
 \liminf_{N\to\infty}\frac1{\mathsf{t}_N^-}
 \log\Prob(X_N+Y_N\le-a_Nx)
 \ge-\frac1{12}(u+2\eta)^3
 -\frac{\theta(x-u)^2}{2\sigma^2}.
\]
At $u=0$, use
$Y_N/a_N\in(-x-2\eta,-x-\eta)$ and require $X_N\le\eta a_N$; by
\eqref{eq:conditional-typical-left}, this requirement has probability
$1-o(1)$ uniformly on $\cG_N$.  Hence \eqref{eq:Y-good-lower} gives
\[
 \liminf_{N\to\infty}\frac1{\mathsf{t}_N^-}
 \log\Prob(X_N+Y_N\le-a_Nx)
 \ge-\frac{\theta(x+\eta)^2}{2\sigma^2}.
\]
Optimization and $\eta\downarrow0$ prove (L2).

\medskip
\noindent\emph{5. Left Tracy--Widom regime: $\theta_N^-\to\infty$.}
As on the right, the Gaussian displacement is too expensive to contribute at
the Tracy--Widom speed, so the proof localizes $Y_N/a_N$ near zero.  For
$0<\eta<x$,
\begin{align*}
 \Prob(X_N+Y_N\le-a_Nx)
 &\le
 \Prob(Y_N<-\eta a_N)
 +\E\left[
 \one_{\cG_N}
 \Prob(X_N\le-a_N(x-\eta)\mid\cF_N)
 \right]
 +\Prob(\cG_N^c),                                      \\
 \Prob(X_N+Y_N\le-a_Nx)
 &\ge
 \E\left[
 \one_{\{|Y_N|\le\eta a_N\}\cap\cG_N}
 \Prob(X_N\le-a_N(x+\eta)\mid\cF_N)
 \right].
\end{align*}
The MDP upper bound and
$\mathsf{g}_N/\mathsf{t}_N^-=\theta_N^-\to\infty$ imply
\[
 \limsup_{N\to\infty}\frac1{\mathsf{t}_N^-}
 \log\Prob(|Y_N|>\eta a_N)=-\infty.
\]
Thus the Gaussian displacement is superexponentially negligible at speed
$\mathsf{t}_N^-$.  Together with the good-event estimate, this yields
\[
 \frac1{\mathsf{t}_N^-}\log
 \Prob(|Y_N|\le\eta a_N,\ \cG_N)\longrightarrow0.
\]
The two uniform conditional left-tail estimates, followed by
$\eta\downarrow0$, prove (L1).

\medskip
\noindent\emph{6. Left Gaussian regime: $\theta_N^-\to0$.}
At speed $\mathsf{g}_N$, a moderate left Airy fluctuation is
superexponentially unlikely.  A negative displacement of order $-a_Nx$ gives
the leading contribution, while $X_N$ remains conditionally typical.
For $0<\eta<x$,
\begin{align}
 \Prob(X_N+Y_N\le-a_Nx)
 &\le
 \Prob(Y_N\le-a_N(x-\eta))
 +\E\left[
 \one_{\cG_N}\Prob(X_N\le-\eta a_N\mid\cF_N)
 \right]
 +\Prob(\cG_N^c).                                      \label{eq:left-G-upper}
\end{align}
Uniformly on $\cG_N$, the second term is at most
\[
 \exp\left\{-\mathsf{t}_N^-
 \left(\frac1{12}\eta^3+o(1)\right)\right\}.
\]
Since $\mathsf{t}_N^-/\mathsf{g}_N=1/\theta_N^-\to\infty$, this term and the
good-event complement are superexponentially small at speed $\mathsf{g}_N$.
The Gaussian MDP therefore gives the upper bound
$-(x-\eta)^2/(2\sigma^2)$.

For the lower bound, restrict to
$Y_N/a_N\in(-x-2\eta,-x-\eta)$ and require $X_N\le\eta a_N$:
\begin{align}
 \Prob(X_N+Y_N\le-a_Nx)
 &\ge
 \E\left[
 \one_{\{Y_N/a_N\in(-x-2\eta,-x-\eta)\}\cap\cG_N}
 \Prob(X_N\le\eta a_N\mid\cF_N)
 \right].                                             \label{eq:left-G-lower}
\end{align}
Equation \eqref{eq:conditional-typical-left} gives a deterministic sequence
$\delta_N\downarrow0$ such that the conditional factor in
\eqref{eq:left-G-lower} is at least $1-\delta_N$ uniformly on $\cG_N$. 
The uniform conditional estimate is essential here because
$\Prob(X_N>\eta a_N)$ may exceed the Gaussian left-tail probability in this
regime.  Applying the MDP lower bound to
\eqref{eq:left-G-lower} gives the lower exponent
$-(x+\eta)^2/(2\sigma^2)$; sending $\eta\downarrow0$ proves (L3).
\end{proof}

We now verify that the abstract convolution principle applies to the present
random matrix model.

\begin{proof}[Proof of \Cref{thm:main}]
\Cref{sec:shift,sec:kernel,sec:airy,sec:conditional-tails}
concern the truncated ensemble, with the hats suppressed after \Cref{lem:truncation-coupling}.
We first apply the abstract lemma to that ensemble.  Use the exact
decomposition \eqref{eq:Z-decomposition} with $X_N=\chi_N$ and $Y_N=s_N$.
By \Cref{prop:conditional-TW}, $X_N$ satisfies the conditional right and
left estimates in \Cref{lem:conditional-convolution}, uniformly on
$\cG_N$.  By \Cref{prop:shift-mdp}, with
$b_N=\alpha_N a_N$,
\[
 \frac{s_N}{a_N}=\frac{U_N}{\alpha_N a_N}
\]
satisfies an MDP with speed $\alpha_N^2a_N^2$ and rate
$v^2/(2\sigma^2)$.  The hypotheses on $b_N$ hold because
$\alpha_N a_N\to\infty$ and
$\alpha_N a_N=N^{o(1)}=o(\sqrt N)$.  Finally,
\Cref{lem:good-event} makes $\cG_N^c$ superexponentially small at all
relevant speeds.  The six limits for the truncated ensemble follow from
\Cref{lem:conditional-convolution}.

It remains to return to the original matrix.  Define
\[
 \widehat Z_N=\frac{\sqrt N}{\alpha_N}(\widehat\lambda_N-R_N).
\]
By \Cref{lem:truncation-coupling}, the original and truncated largest
eigenvalues coincide outside an event that is superexponentially small at
every speed appearing in the six limits.  For each one-sided Borel set
$\mathcal E_N$ used above,
\[
 \left|\Prob(Z_N\in \mathcal E_N)-
 \Prob(\widehat Z_N\in \mathcal E_N)\right|
 \le\Prob(\lambda_N\ne\widehat\lambda_N).
\]
To make the last step explicit, fix one of the six regimes, denote its
speed by $\upsilon_N$, and let $\pi_N$ and $\widehat\pi_N$ be the corresponding
one-sided probabilities for the original and truncated matrices.  The
truncated result gives
$\log\widehat\pi_N/\upsilon_N\to-I$ for a finite $I>0$.  For every $M>I+1$,
\Cref{lem:truncation-coupling} gives, for all large $N$,
\[
 |\pi_N-\widehat\pi_N|\le e^{-M\upsilon_N},
 \qquad
 \widehat\pi_N\ge e^{-(I+1)\upsilon_N}.
\]
Hence
$|\pi_N/\widehat\pi_N-1|\le e^{-(M-I-1)\upsilon_N}\to0$, and therefore
$\log\pi_N/\upsilon_N\to-I$.  This proves the same six limits for the original
matrix $M_N$.
\end{proof}

The fixed-coupling corollary is an immediate reading of the two limiting
transition parameters.

\begin{proof}[Proof of \Cref{cor:fixed-alpha}]
If $\alpha_N\equiv\alpha$, then
$\theta_N^+=\alpha^2a_N^{1/2}\to\infty$, so
\eqref{eq:fixed-right} is (R1).  Also
$\theta_N^-=\alpha^2/a_N\to0$, and (L3) gives
\[
 \frac1{\alpha^2a_N^2}
 \log\Prob(Z_N\le-a_Nx)
 \longrightarrow-\frac{x^2}{2\sigma^2},
\]
which is equivalent to \eqref{eq:fixed-left}.
\end{proof}


\section{Concluding remarks}\label{sec:conclude-rmk}


We end by summarizing the main results and pointing out several directions for future work.

\smallskip

The phase diagram results from the interaction of two sources of
fluctuation.  Given the
diagonal disorder, the largest eigenvalue is controlled by a cubic Airy
saddle and has Tracy--Widom moderate tails.  On the $Z_N$ (Airy) scale, the disorder itself moves the
conditional edge by a resolvent statistic whose leading term is
$\alpha_N^{-1}N^{-1/2}\sum_j\xi_j$ and therefore has Gaussian moderate
deviations.  The asymmetry between the Tracy--Widom exponents $3/2$ and $3$
produces the two different transition scales,
$\alpha_N\asymp a_N^{-1/4}$ on the right and
$\alpha_N\asymp a_N^{1/2}$ on the left.  In particular, fixed coupling has a
Tracy--Widom right tail but a Gaussian left tail.

The proof isolates three components that may be useful in related problems.
First, the truncation coupling is superexponentially accurate at every
moderate speed in the theorem.  Second, the resolvent expansion identifies
the random edge displacement with a normalized sum and yields its Gaussian
MDP.  Third, the conditional kernel analysis gives a trace-norm Airy
approximation uniformly over good diagonal configurations.  The
conditional convolution argument in \Cref{sec:phase-proof} is based on this
uniformity.

The logarithmic restriction $a_N^3=o(\log N)$ is a limitation of the
quantitative transfer from the finite-$N$ conditional kernel to the very
small left Fredholm determinant.  This restriction arises at the final
relative determinant comparison: a polynomially accurate trace-norm estimate is
negligible relative to $F_2(-a_Nx)$ only in the logarithmic window.
Extending the theorem to a regime where $a_N^3$ is comparable to, or larger
than, $\log N$ would require a much sharper finite-$N$ expansion, uniform
over the random good set, with an error that is exponentially small on the
left-tail scale.

Several closely related directions seem natural.  First, one may try to
remove the subpolynomial condition on $\alpha_N$ and allow polynomially
growing or vanishing coupling; this would require new truncation, saddle,
and contour estimates.  Second, it would be interesting to develop the
Gaussian orthogonal ensemble (GOE),
$\beta$-ensemble, or non-invariant analogues of the present phase diagram,
where the conditional determinantal structure must be replaced by Pfaffian or
universality arguments.  Third, weakening the Cram\'er condition on the
diagonal entries could replace the Gaussian displacement MDP by
different extreme-value behavior and may produce new transition laws.  These
directions retain the same basic interaction between an Airy fluctuation
and a random motion of the edge, but each one changes a different technical
part of the present proof.

\appendix


\section{Proof of Lemma~\ref{lem:shifted-contours}}\label{app:shifted-contours}


\begin{proof}

The principal difficulty is that the finite-$N$ phase $\Phi_N$ is defined
through the random resolvent $f_N$, whereas the Airy limit involves the pure
cubic $\zeta^3/3$.  The contour separation parameter $d$ shifts the
integration rays away from the saddle, and we must verify that the
displaced rays retain the decay obtained in \Cref{lem:global-decay} for the
unshifted rays $C_1,\dots,C_4$.

We give the argument on $\Gamma^+$; the other three branches are completely
analogous, as we indicate at the end of each step.  Write
\[
 r_0(t)=t(1+{\rm i}),
 \qquad
 r(t)=d+r_0(t)=d+t+{\rm i}t,
\]
so that $r(t)$ is the parametrisation of the upper ray of $\Gamma^+$.
Let $c_0$ be the window exponent in \Cref{lem:cubic-expansion}, and denote by
$c_1>0$ the exponent of $N$ appearing in its error bound \eqref{eq:cubic-error}.
We fix once and for all an exponent
\[
 0<\eta<\min\{c_0,\;c_1/5,\;1/12\}.
\]
The reason for each constraint will become clear below: $c_0$ guarantees that
the cubic expansion is valid on the local window, $c_1/5$ ensures that the
error term $CN^{-c_1}(1+t^4)$ is $o(1)$ uniformly for $t\le N^{\eta}$, and
$1/12$ guarantees that the horizontal displacement is small compared with the
natural cubic scale.

\medskip
\noindent{{\it Step~1. The local window: $0\le t\le N^{\eta}$}.}

On this range the scaled increment $h=\ell_N r(t)$ satisfies
$|h|/(w_N-\tau_N)\le N^{\eta-1/3+o(1)}=o(1)$, so the Taylor expansion used in
\Cref{lem:cubic-expansion} is valid.  A direct computation gives the exact
algebraic identity
\[
 \Re\frac{(d+t+{\rm i}t)^3}{3}
 =\frac{d^3}{3}+d^2t-\frac23t^3.
\]
Note that $d^2t\leq(2d^3+t^3)/3$ for any $d>0,t>0$, so there exist constants
$c,C>0$ (depending only on $d$) such that
\[
 \Re\frac{(d+t+{\rm i}t)^3}{3}\le C-ct^3,
 \qquad t\ge0.
\]
Together with the uniform approximation \eqref{eq:cubic-error} from
\Cref{lem:cubic-expansion}, which asserts that $|\Phi_N(r)-r^3/3|$ is
uniformly small on this window, we immediately obtain
\[
 \Re\Phi_N(r(t))\le C-ct^3,
 \qquad 0\le t\le N^{\eta},
\]
which is the local piece of \eqref{eq:shifted-r-local}.

To prove the comparison \eqref{eq:shifted-r-comparison}, set
$\Delta_N(r)=\Phi_N(r)-r^3/3$.  By \eqref{eq:cubic-error} and the choice
$\eta<c_1/5$,
\[
 |\Delta_N(r(t))|
 \le CN^{-c_1}(1+t^4)
 \le CN^{-c_1}(1+N^{4\eta})
 =o(1),
\]
uniformly for $0\le t\le N^{\eta}$.  Hence, factoring out the Airy cubic,
\begin{align*}
 |e^{\Phi_N(r)}-e^{r^3/3}|
 &=e^{\Re(r^3/3)}|e^{\Delta_N(r)}-1|\\
 &\le e^{\Re(r^3/3)}\,|\Delta_N(r)|\,e^{|\Delta_N(r)|}\\
 &\le CN^{-c_1}(1+t^4)e^{-ct^3},
\end{align*}
where in the second line we used $|e^z-1|\le|z|e^{|z|}$ and in the third line
we absorbed the harmless factor $e^{o(1)}$ into the constant $C$.  This proves
\eqref{eq:shifted-r-comparison}.  The corresponding statements
\eqref{eq:shifted-s-local} and \eqref{eq:shifted-s-comparison} on $\Sigma$
follow from the identity
\[
 -\Re\frac{(-d-t+{\rm i}t)^3}{3}
 =\frac{d^3}{3}+d^2t-\frac23t^3,
\]
combined with the same error estimate for $\Phi_N$ on the left contours.

\medskip
\noindent{{\it Step~2. The mesoscopic range: $N^{\eta}\le t\le T_N^*$}.}

On this intermediate scale the cubic approximation of \Cref{lem:cubic-expansion}
is no longer pointwise accurate, but the unshifted rays $C_1,\dots,C_4$ already
carry strong decay by \Cref{lem:global-decay}.  Our task is to prove
that this decay persists under the horizontal shift by $d\ell_N$.  The key tool is a
Taylor expansion of the derivative $f_N'$ that controls the change of the real
part of the phase when we move horizontally from the unshifted point
$w_N+\ell_N r_0(t)$ to the shifted point $w_N+\ell_N r(t)$.

Let $\zeta$ be a complex number with $|\zeta|\le 3T_N^*/2$ and let
$u\in[0,d]$.  Taylor's formula with integral remainder for $f_N'$ at
$w_N$ gives
\begin{align}
 \ell_N f_N'(w_N+\ell_N(\zeta+u))
 &=\ell_N^2f_N''(w_N)(\zeta+u)\nonumber\\
 &\quad+\ell_N^3(\zeta+u)^2
 \int_0^1(1-v)
 f_N^{(3)}(w_N+v\ell_N(\zeta+u))\,\dd v.
 \label{eq:app-scaled-derivative-bound-start}
\end{align}
All points $w_N+v\ell_N(\zeta+u)$ appearing in the integrand stay at a fixed
positive fraction of $w_N$ away from the poles $\xi_j$.
Indeed, since $|\zeta|\le3T_N^*/2$ and $u\in[0,d]$,
\[
 \ell_N(|\zeta|+u)
 \le \frac34w_N+d\ell_N.
\]
Consequently, uniformly for $v\in[0,1]$ and $1\le j\le N$,
\[
 \bigl|w_N+v\ell_N(\zeta+u)-\xi_j\bigr|
 \ge \frac14w_N-d\ell_N-\tau_N
 \ge c w_N,
\]
because $\ell_N/w_N=N^{-1/3}(1+o(1))$ and $\tau_N/w_N\to0$. 
Consequently, \eqref{eq:fhigher} gives
$|f_N^{(3)}(z)|\le CN/w_N^3$ uniformly on $\cG_N$.

We now estimate the two terms on the right-hand side of
\eqref{eq:app-scaled-derivative-bound-start} separately.  For the linear term,
\Cref{lem:cubic-expansion} (more precisely \eqref{eq:quadratic-scaled-bound})
gives $|\ell_N^2 f_N''(w_N)|\le CN^{-c}$ for some $c>0$.  For the quadratic
term, we use the uniform bound on $f_N^{(3)}$ and the relation
$N\ell_N^3/w_N^3=1+o(1)$ to obtain
\[
 \bigl|\ell_N^3(\zeta+u)^2 f_N^{(3)}(\cdot)\bigr|
 \le C|\zeta+u|^2.
\]
Combining these bounds yields the uniform estimate
\begin{equation}\label{eq:app-scaled-derivative-bound}
 \bigl|\ell_N f_N'(w_N+\ell_N(\zeta+u))\bigr|
 \le C\bigl(N^{-c}|\zeta|+1+|\zeta|^2\bigr).
\end{equation}

We integrate the derivative \eqref{eq:app-scaled-derivative-bound-start} along
the horizontal segment from $r_0(t)$ to $r(t)=r_0(t)+d$.  Since the length of
this segment is the fixed constant $d$, we obtain
\begin{equation}\label{eq:app-horizontal-change-meso}
 \bigl|\Re\Phi_N(r(t))-\Re\Phi_N(r_0(t))\bigr|
 \le C\bigl(N^{-c}t+1+t^2\bigr)
 \le C'(1+t^2).
\end{equation}
Here we used $|\zeta|\le|r_0(t)|+|u|\le Ct$ on the segment.

By the unshifted decay estimate \eqref{eq:unshifted-local-decay} from
\Cref{lem:global-decay}, we have $\Re\Phi_N(r_0(t))\le-ct^3$.  Therefore
\[
 \Re\Phi_N(r(t))
 \le -ct^3 + C'(1+t^2).
\]
Because $t\ge N^{\eta}$ and $\eta>0$, the cubic term dominates the quadratic
error for all sufficiently large $N$: indeed,
\[
 \frac{t^3}{1+t^2}\ge\frac{N^{3\eta}}{1+N^{2\eta}}\to\infty,
\]
so $-ct^3+C'(1+t^2)\le -(c/2)t^3$ for large $N$.  This establishes
\eqref{eq:shifted-r-local} on the mesoscopic range.  The same argument,
applied to the left rays and using \eqref{eq:gminus-prime}, yields the
analogous bound \eqref{eq:shifted-s-local} on $\Sigma$.

\medskip
\noindent{{\it Step~3. The far range: $t\ge T_N^*$}.}

Beyond the cubic-decay scale $T_N^*=w_N/(2\ell_N)\asymp N^{1/3}$, the
unshifted phase has already switched to linear decay.  Write $t=T_N^*+h$ with
$h\ge0$.  From \eqref{eq:unshifted-far-decay} we inherit the unshifted bound
\begin{equation}\label{eq:app-scaled-unshifted-far}
 \Re\Phi_N(r_0(t))
 \le-cN-cN^{2/3}h.
\end{equation}
We must again control the horizontal change from $r_0(t)$ to $r(t)$.  On the
segment joining these two points, a direct estimate from the definition
\eqref{eq:phase} of $f_N$ gives
\begin{equation}\label{eq:app-scaled-derivative-far}
 \bigl|\ell_N f_N'(w_N+\ell_N(\zeta+u))\bigr|
 \le C\bigl(N^{2/3}+N^{1/3}t\bigr).
\end{equation}
Indeed, the quadratic part $(w-v_N)/S_N$ contributes at most
$C(N^{2/3}+N^{1/3}t)$ because $v_N=O(w_N)$ uniformly on $\cG_N$
by \eqref{eq:vN-bounds} and
$|w|\le w_N+\ell_N(|r_0(t)|+d)\le C(w_N+\ell_N t)$.
For the logarithmic part, the needed denominator bound follows
directly from the contour geometry.  On the right horizontal segments,
$\Re w\ge w_N$, so $|w-\xi_j|\ge w_N-\tau_N\ge c w_N$.  On the left
horizontal segments, writing $q=\ell_Nt\ge w_N/2$, one has
$|\Im w|=q$ and hence $|w-\xi_j|\ge q\ge w_N/2$.  Therefore, on every far
horizontal segment,
\[
 \ell_N\left|\sum_{j=1}^N\frac1{w-\xi_j}\right|
 \le C\frac{\ell_NN}{w_N}
 =CN^{2/3}(1+o(1)).
\]

Integrating \eqref{eq:app-scaled-derivative-far} over the horizontal segment of
fixed length $d$ changes the real part by at most
\[
 C\bigl(N^{2/3}+N^{1/3}T_N^*+N^{1/3}h\bigr)
 \le C'N^{2/3}+C'N^{1/3}h.
\]
The first term $C'N^{2/3}$ is absorbed by the $-cN$ term in
\eqref{eq:app-scaled-unshifted-far} (since $N^{2/3}=o(N)$).  The second term
$C'N^{1/3}h$ is absorbed by $-cN^{2/3}h$ after decreasing the constant $c>0$,
because, with $t$ here denoting the scaled contour parameter, the
linear term in \eqref{eq:unshifted-far-decay} becomes
\[
 -c\frac{N}{w_N}\left(\ell_Nt-\frac{w_N}{2}\right)
 =-c\frac{N\ell_N}{w_N}(t-T_N^*)
 =-cN^{2/3}(1+o(1))(t-T_N^*),
\]
whose coefficient dominates any fixed multiple of $N^{1/3}$.  Hence, for all
sufficiently large $N$,
\[
 \Re\Phi_N(r(t))
 \le -\frac{c}{2}N - \frac{c}{2}N^{2/3}h,
\]
which proves \eqref{eq:shifted-r-far}.  The estimates
\eqref{eq:shifted-s-far} on $\Sigma$ follow by the same argument applied to
the left rays, using \eqref{eq:gminus-prime} in place of
\eqref{eq:gplus-prime}.

\medskip
\noindent{{\it Step~4. Uniform tail integrals}.}

It remains to prove the exponentially small bounds
\eqref{eq:r-tail-integral}--\eqref{eq:s-tail-integral}.  We discuss only
\eqref{eq:r-tail-integral}; the other integral is analogous.

First observe that the factor $e^{-\rho(d+t)}$ appearing in the integrand is
bounded by $e^{-d\rho}e^{-\rho t}\le e^{-d\rho}$, so $e^{-d\rho}$ may be
extracted from the integral.  The remaining $\rho$-dependent contribution is
nonpositive and therefore only improves the estimate.

We now split the domain of integration at $T_N^*$.

\smallskip
\noindent{{\it The mesoscopic tail: $N^{\eta}\le t\le T_N^*$}.} \ On this range, \eqref{eq:shifted-r-local} gives $\Re\Phi_N(r(t))\le C-ct^3$.
With a spatial variable $q\in[0,L]$ contributing $q(d+t)$ in the exponent,
the integrand is bounded by
\[
 (1+t^m)\exp\{-ct^3+q(d+t)\}.
\]
Because $L^3\le c\log N$, we have $L=o(N^{2\eta})$. Note that,  for $t\ge N^\eta$,
\[
\frac{Lt}{t^3}\le \frac{L}{N^{2\eta}}=o(1) \quad  \text{and} \quad \frac{Ld}{t^3}\le \frac{Ld}{N^{3\eta}}=o(1),
\]
consequently, uniformly
for $0\le q\le L$ and $t\ge N^{\eta}$,
\[
 q(d+t)\le L(d+t)\le \frac{c}{2}t^3,
 \qquad\text{for all large }N.
\]
The integrand is therefore bounded by
$(1+t^m)e^{-(c/2)t^3}$, and its integral over $[N^{\eta},T_N^*]$ is dominated
by $C_m e^{-c'N^{3\eta}}$ for a suitable $c'>0$.

\smallskip
\noindent{{\it The far tail: $t\ge T_N^*$}.} \ Write $t=T_N^*+h$ with $h\ge0$.  By \eqref{eq:shifted-r-far}, the exponent is
bounded by
\[
 -cN-cN^{2/3}h+q(d+t)-\rho(d+t).
\]
Since $q\le L$ and $L^3\le c\log N$, we have $L=o(N)$ and $L=o(N^{2/3})$.
Hence the positive spatial contribution $q(d+t)\le L(d+T_N^*+h)$ satisfies
\[
 LT_N^*=o(N)
 \quad\text{and}\quad
 Lh=o(N^{2/3}h).
\]
Both terms are therefore absorbed by the corresponding negative terms
$-cN$ and $-cN^{2/3}h$ for all large $N$.  The resulting integral over
$h\in[0,\infty)$ contributes at most $C_m e^{-cN}$, which is negligible
compared with $e^{-cN^{3\eta}}$.

\smallskip
Combining the two ranges and extracting the prefactor $e^{-d\rho}$ yields
\eqref{eq:r-tail-integral}.  The Airy counterparts (with $\Phi_N$ replaced by
$\zeta^3/3$) are simpler because the cubic decay $-ct^3$ holds for all
$t\ge0$, so the same argument gives an even stronger bound.  Finally, if a
nonnegative spatial variable $x$ or $y$ is present, it contributes an
additional factor $e^{-dx}$ or $e^{-dy}$ that can be extracted immediately,
while the term $q(d+t)$ is omitted because $x,y\ge0$ no longer produces a
positive exponential.  This completes the proof.
\end{proof}


\section*{Acknowledgements}


The authors are grateful to Lu Zhang for helpful comments on an earlier version of this paper.


\end{document}